\documentclass[12pt]{amsart} 

\usepackage{amsmath,wasysym}
\usepackage{amsxtra}
\usepackage{amscd}
\usepackage{amsthm}
\usepackage{amsfonts}
\usepackage{amssymb}
\usepackage{eucal}
\usepackage{graphics}
\usepackage{hyperref,mathrsfs}
\usepackage[usenames,dvipsnames]{xcolor}
\usepackage{tikz}
\usepackage{tikz-cd}
\usepackage{stmaryrd}
\usepackage[all]{xy}
\usepackage{hyperref}
\usepackage{color}
\usepackage{bbm} 
\usepackage{upgreek}
\usepackage{extarrows}
\usepackage{yhmath}
\usepackage{thmtools}
\usepackage[capitalise]{cleveref}
\crefname{equation}{}{}

\allowdisplaybreaks

\hypersetup{colorlinks,linkcolor=blue, urlcolor=blue,
	citecolor=blue}

\theoremstyle{plain}
\newtheorem{thm}{Theorem}[section]
\newtheorem{lem}[thm]{Lemma}
\newtheorem{cor}[thm]{Corollary}
\newtheorem{prop}[thm]{Proposition}

\theoremstyle{definition}
\newtheorem{definition}[thm]{Definition}

\newtheorem{rem}[thm]{Remark}

\numberwithin{equation}{section}

\newcommand{\C}{{\mathbb C}}

\newcommand{\co}{\text{co}}
\newcommand{\ro}{\text{ro}}

\DeclareMathOperator{\pt}{pt}

\DeclareMathOperator{\Span}{Span}
\DeclareMathOperator{\Sp}{Sp}
\DeclareMathOperator{\AIII}{AIII}

\DeclareMathOperator{\Eu}{Eu}
\DeclareMathOperator{\diag}{diag}
\DeclareMathOperator{\Gr}{Gr}

\DeclareMathOperator{\Res}{Res}

\begin{document}

\title[Twisted Yangians and Steinberg varieties of type C]{Twisted Yangians of type AIII and Steinberg varieties of type C}

\author{Changjian Su}
\address[Changjian Su]{Yau Mathematical Sciences Center, Tsinghua University, Beijing, China}
\email{changjiansu@mail.tsinghua.edu.cn}

\author{Yang Yang}
\address[Yang Yang]{School of Mathematical Sciences, Key Laboratory of MEA (Ministry of Education) \& Shanghai Key Laboratory of PMMP, East China Normal University, Shanghai 200241, China}
\email{52275500011@stu.ecnu.edu.cn}

\begin{abstract}
    We study the equivariant homology of the generalized Steinberg variety of type C and show that there exists a surjective algebra homomorphism from the twisted Yangian of type $\AIII_{2n}^{(\tau)}$ to it. 
\end{abstract}

\maketitle

\setcounter{tocdepth}{1}
\tableofcontents

\section{Introduction}

Geometric methods provide powerful tools for studying representations of algebras. A key step is to realize these algebras geometrically via convolution; one can then use representations of the resulting convolution algebras to construct representations of the algebras of interest, see \cite{CG97}. Kazhdan and Lusztig used the equivariant K-group of the Steinberg variety to give a geometric realization of the affine Hecke algebra for the Langlands dual group and classified its irreducible representations via the Springer fibers \cite{KL87} (also see \cite{CG97, X07}). This also works for the quantum affine algebras. Ginzburg--Vasserot and Vasserot studied the affine quantum group for $\mathfrak{gl}_n$ via the equivariant K-group of the generalized Steinberg varieties corresponding to the $n$-step partial flag variety in type A \cite{GV93,V98}. The cotangent bundle of the type A partial flag varieties are special cases of Nakajima quiver varieties. Nakajima used the equivariant K-theory of his quiver varieties to realize the quantum affine algebras of type ADE \cite{Na01}, and Varagnolo used the equivariant homology of the quiver varieties to realize the corresponding Yangians \cite{Vara00}. The Nakajima quiver varieties are Higgs branch of the quiver gauge theories. Recently, Bravermann, Finkelberg and Nakajima gave a mathematically rigorous definition of the Coulomb branches \cite{BFN18}. Moreover, the shifted Yangian and the shifted quantum affine algebras can be realized via the quantized Coulomb branches of the corresponding quiver gauge theories, see \cite{BFN19,FT19,NW23,We19}. In all these works, the Drinfeld new presentation for the quantum algebras play an important role \cite{Dr87, Beck, Da12}. On the other hand, the quantum group of type A can also be realized by the type A partial flag varieties over a finite field by Beilinson--Lusztig--MacPherson \cite{BLM}. Through the Maulik--Okounkov stable envelopes, the Yangians and the quantum affine algebras can also be realized via the geometry of the Nakajima quiver varieties by the RTT formalism, see \cite{MO19,OS22}

iQuantum groups, arising from the quantization of symmetric pairs $(\mathfrak{g},\mathfrak{g}^\sigma)$ into quantum symmetric pairs, form a natural generalization of Drinfeld–Jimbo quantum groups.
Many important constructions for the quantum groups have been generalized to the iquantum groups, see \cite{W23}. In \cite{BKLW18}, the iquantum group of quasi-split type AIII was realized by counting points over finite fields of the $N$-step isotropic flags of type B, and this has been generalized in \cite{FLW+} to the affine case. Thanks to the Drinfeld-type new presentation of the affine 
iquantum group of type AIII obtained recently in \cite{LWZ24}, this affine iquantum group can be realized via the equivariant K-theory of generalized Steinberg varieties of type  C \cite{SW24} (also cf. \cite{FMX22}), generalizing the work of Ginzburg--Vasserot and Vasserot \cite{GV93,V98}. The cotangent bundle of partial flag varieties of classical types are examples of the $\sigma$-quiver varieties defined by Li \cite{Li19}, which are introduced to study the geometric representation theory of the (quantum) symmetric pairs. In particular, using the sheaf-theoretic formulation of Maulik–Okounkov stable envelopes \cite{Na17}, Li constructed stable envelopes for the $\sigma$-quiver variety, and consequently defined the twisted Yangian—a coideal subalgebra of the Maulik–Okounkov Yangian—which acts on the cohomology of quiver varieties. Building upon \cite{LWZ25b,LWZ25a}, Lu and Zhang established the Drinfeld new presentation for the twisted Yangian of quasi-split type \cite{LZ25}. Based on this, we can use the Coulomb branch of the involution fixed part of the quiver gauge theory to give a geometric realization of the shifted twisted Yangian \cite{LWW25a,SSX25} (also cf. \cite{bartlett2025gklo,LWW25b,wang2025quivers}). In \cite{dong2024flag}, Dong and Ma studied the equivariant homology of the generalized Steinberg varieties of type B and C, defined a pre-twisted Yangian, and constructed an algebra homomorphism from the pre-twisted Yangian to the equivariant homology the generalized Steinberg varieties. The pre-twisted Yangian is closely related to the twisted Yangian of type AIII. Moreover, Luo, Xu and the second named author used the generalized Steinberg variety to give a geometric construction of the Schur algebra \cite{LXY24}.

Now let us state the main result of this paper. Let $d$ and $n$ be two positive integers, and let $N=2n+1$. Let $V :=\C^{2d}$ be a symplectic vector space, $G = \Sp(V)$, and $\mathscr{F}$ be the $N$-step partial flag variety with cotangent bundle $\mathscr{M}:=T^*\mathscr{F}$. The torus $\mathbb{C}^*$ acts on $\mathscr{M}$ by scaling the cotangent fibers, and let $\hbar$ be its equivariant parameter. The natural $G$-action on $\mathscr{F}$ induces a Hamiltonian action on the symplectic variety $\mathscr{M}$, with moment map $\mathscr{M}\rightarrow \mathscr{N}\hookrightarrow \mathfrak{sp}_{2d}$, where $\mathscr{N}$ is the nilpotent cone of the Lie algebra $\mathfrak{sp}_{2d}$. Let $Z:=\mathscr{M}\times_{\mathscr{N}}\mathscr{M}$ be the generalzied Steinberg variety, and $H_*^{G\times \mathbb{C}^*}(Z)$ be the equivariant Borel--Moore homology with complex coefficients. Via convolution, it is an associative algebra over the base ring $H_*^{G\times \mathbb{C}^*}(\pt)\simeq \mathbb{C}[\mathfrak{g}]^G[\hbar]$, see \cite{CG97,AF24}. Let $\mathbf{Y}^\imath$ be the twisted Yangian of quasi-split type $\AIII_{2n}^{(\tau)}$, which is associated to the following Satake diagram: 
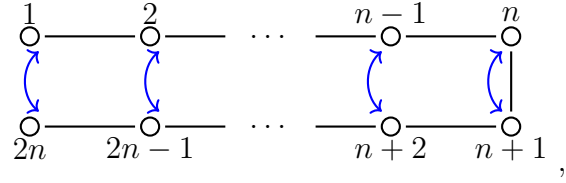
\begin{center}
\begin{tikzpicture}[scale=.4]
\node at (0,0.75) {$1$};
\node at (4,0.75) {$2$};
\node at (12,0.75) {$n-1$};
\node at (16,0.75) {$n$};     
\node at (0,-3.75) {$2n$};
\node at (4,-3.75) {$2n-1$};
\node at (12,-3.75) {$n+2$};
\node at (16,-3.75) {$n+1$}; 
       
\node at (8,0) {$\dots$};
\node at (8,-3) {$\dots$};
    
\foreach \x in {0,2,6,8} 
{\draw[thick,xshift=\x cm] (\x, 0) circle (0.3); 
\draw[thick,xshift=\x cm] (\x, -3) circle (0.3);}

\foreach \x in {0,6}
{\draw[thick,xshift=\x cm] (\x,0) ++(0.5,0) -- +(3,0);
\draw[thick,xshift=\x cm] (\x,-3) ++(0.5,0) -- +(3,0);}
   
\foreach \x in {2,4.5}
{\draw[thick,xshift=\x cm] (\x,0) ++(0.5,0) -- +(2,0);
\draw[thick,xshift=\x cm] (\x,-3) ++(0.5,0) -- +(2,0);}
    
\foreach \x in {8}
\draw[thick,xshift=\x cm] (\x,-0.5) -- +(0,-2);
    
\foreach \x in {0,4} 
\draw[thick,<->, blue, bend left=50] (\x+0.3,-2.5) to (\x+0.3,-0.5);
\foreach \x in {12,16} 
\draw[thick,<->, blue, bend left=50] (\x-0.3,-2.5) to (\x-0.3,-0.5);
 
\end{tikzpicture},
\end{center}
where the blue arrows indicate the diagram involution $\tau$.
The following is our main result. 
\begin{thm}[Theorem \ref{thm:main} and \ref{surj}]\label{thm:mainintro}
    There exists an algebra homomorphism $$\Psi:\mathbf{Y}^\imath\rightarrow H_*^{G\times \C^*}(Z),$$ and it is surjective if we specialize to a semisimple element $(s,t)\in \mathfrak{g}\times \mathbb{C}$ with $t\neq 0$.
\end{thm}

The idea of the proof is as follows. For generators of the twisted Yangian $\mathbf{Y}^\imath$, we associate some explicit elements $\mathscr{H}_{i,r}$ and $\mathscr{B}_{i,r}$ in the equivariant homology group. In order to show that this is an algebra homomorphism, we check the relations on the faithful module $H_*^{G\times \C^*}(\mathscr{F})$ of $H_*^{G\times \C^*}(Z)$, see Lemma \ref{lemma:faithful}. In \cite{dong2024flag,FMX22,SW24}, the operators are divided into two kinds, hence in checking the relations there are many cases. The first advantage of our approach is that we have a uniform formula for all the operators. Moreover, there is an obvious involution on $H_*^{G\times \C^*}(Z)$ by switching the two factors of $\mathscr{M}$ in the Steinberg variety $Z$, which behaves well with respect to the convolution and interchanges the elements $\mathscr{B}_{i,r}$ corresponding to the generators of $\mathbf{Y}^\imath$, see Lemma \ref{lem:inv}. These two observations significantly simplified the verifications of the relations. 

Using the construction of the representations of the convolution algebras in \cite[Chapter 8]{CG97}, we are able to construct some representations of the twisted Yangian $\mathbf{Y}^\imath$. On the other hand, the twisted Yangian $\mathbf{Y}^\imath$ studied in this paper is closely related to the reflection algebra introduced by Molev and Ragoucy, who also gave a complete description of the finite-dimensional irreducible representations of the latter, see \cite[Theorem 6.17]{LZ25} and \cite{MR02}. In particular, the shifts $\frac{n-i}{2}\hbar+\frac{\hbar}{4}$ in our generators $\mathscr{H}_{i,r}$ and $\mathscr{B}_{i,r}$ in \eqref{equ:geoHir} and \eqref{equ:geoBir} below matches well with the shifts in \cite[Equation (6.23) and (6.24)]{LZ25}. It is an interesting question to identify the geometrically constructed representations with those from \cite{MR02}.

An immediate corollary of Theorem \ref{thm:mainintro} is that the twisted Yangian $\mathbf{Y}^\imath$ acts on the equivariant cohomology of the cotangent bundles of the $N$-step partial flag varieties of type C. This result was also obtained by Nakajima in \cite[Theorem 7.4]{nakajima2025instantons} as follows. The cotangent bundles of the partial flag varieties arise as special cases of the $\sigma$-quiver varieties. In \textit{loc.\ cit.}, Nakajima computed the associated K-matrices and thereby identified the twisted Yangian defined via stable envelopes with the Molev–Ragoucy reflection algebra. Consequently, the cohomology of these cotangent bundles carries a representation of the Molev–Ragoucy reflection algebra.

Let us conclude the introduction with an overview of the paper's structure. In Secton \ref{sec:con}, we recall the convolution constructions and prove a generating set for $H_*^{G\times \C^*}(Z)$. In Section \ref{sec:operators}, we construct the explicit homology elements corresponding to the generators in the twisted Yangian $\mathbf{Y}^\imath$. In Section \ref{sec:main}, we recall the Drinfeld new presentation of $\mathbf{Y}^\imath$ and state the main result. In Section \ref{sec:rep}, we show that the specialized homomorphism is surjective and construct representations of $\mathbf{Y}^\imath$ via the geometry of the Spaltenstein varieties. Finally, we prove our main theorem in Section \ref{sec:relations}.

\subsection*{Acknowledgments}
We would like to thank Li Luo, Hiraku Nakajima, Weiqiang Wang, Zheming Xu, and Weinan Zhang for helpful conversations. C.S. is supported by the National Key R\&D Program of China  (No. 2025YFA1017400). Y.Y. is supported by the NSF of China (No. 125B2001).

\section{Convolutions and Steinberg varieties of type C}
\label{sec:con}
\subsection{Convolution in equivariant Borel--Moore homology}
\label{subsec:convolution}

For a connected complex reductive algebraic group $G$ and a quasi-projective $G$-variety $X$, let $H^G_*(X)$ (resp. $H_G^*(X)$) denote the $G$-equivariant Borel--Moore homology (resp. $G$-equivariant cohomology) of $X$ with complex coefficients, which is a module for the ring $H^G_*(\pt)$; see \cite{AF24}. Via the Poincar\'e duality, we have an action map, denoted by a dot, $H_G^*(X)\otimes H^G_*(X)\rightarrow H^G_*(X)$. Any $G$-equivariant vector bundle $E$ on $X$ has an equivariant Chern polynomial $\lambda_z(E):=\sum_i z^i c_i(E)\in H_G^*(X)[z]$, where $c_i(E)$ denotes the $i$-th $G$-equivariant Chern class of $E$. Let $K^G(X)$ be the $G$-equivariant K-group of $X$, see \cite{CG97}. Then the equivariant Chern polynomial $\lambda_z(E)$ only depends on the class of $E$ in $K^G(X)$.

Given three smooth $G$-varieties $M_1,~M_2,~M_3$, let
$$p_{ij}:M_1 \times M_2 \times M_3 \rightarrow M_i\times M_j$$
be the obvious projection maps.
Let $Z_{12}\subseteq M_1 \times M_2 $ and $Z_{23}\subseteq M_2 \times M_3 $ be $G$-stable closed subvarieties.
We denote
 \begin{equation*}
   Z_{12}\circ Z_{23}= p_{13}(p_{12}^{-1}(Z_{12})\cap p_{23}^{-1}(Z_{23})).
 \end{equation*}
If the restriction of $p_{13}$ to $p_{12}^{-1}(Z_{12})\cap p_{23}^{-1}(Z_{23})$ is a proper map,
then we define the convolution product as follows (see \cite{CG97}):
\begin{equation*}
\begin{split}
\star: \ \ H_*^{G}(Z_{12})\otimes H_*^{G}(Z_{23})&\longrightarrow H_*^{G}(Z_{12} \circ Z_{23}), \\
   \gamma_{1} \otimes \gamma_{2} &\mapsto p_{{13}{\ast}}(p_{12}^{\ast}\gamma_{1} \cap p_{23}^{\ast}\gamma_{2}).
\end{split}
\end{equation*}

Let $F_i$ ($i=1,2$) be smooth $G$-varieties, $M_i=T^*F_i$, and $\pi_i$ denote the projections $M_i\rightarrow F_i$ and $s_i$ denote the inclusion of the zero section $F_i$ into $M_i$. The torus $\C^*$ acts on $M_i$ by $z\cdot (x,\xi)=(x,z^{-2}\xi)$, where $x\in F_i$ and $\xi\in T_x^*F_i$. Let $q$ denote the degree one representation of $\mathbb{C}^*$, and $\hbar:=c_1^{\C^*}(q^{-2})$. Then $H_*^{\C^*}(\pt)=\C[\hbar]$.

Let $\mathscr{O}\subset F_1\times F_2$ be a smooth closed $G$-stable sub-variety, and $Z_\mathscr{O}$ denote the conormal bundle $T_\mathscr{O}^*(F_1\times F_2)\subset M_1\times M_2$. Suppose 
the projections $p_{i,\mathscr{O}}:\mathscr{O}\rightarrow F_i$ are smooth fibrations with $p_{1,\mathscr{O}}$ being proper, and we also assume that the first projection $p_1:Z_\mathscr{O}\cap (M_1\times F_2)\rightarrow M_1$ is proper (for example, this holds if $F_2$ is proper). We apply the convolution construction to the setting $M_3=\pt$, $Z_{12}=Z_\mathscr{O}$ and $Z_{23}=s_2(F_2)$. By the assumption of $\mathcal{O}$, 
\[p_1(Z_\mathscr{O}\cap (M_1\times F_2))\subset s_1(F_1).\]
Therefore, the convolution gives morphism
\[H_*^{G\times\C^*}(Z_{\mathcal{O}})\otimes H_*^{G\times\C^*}(F_2)\rightarrow H_*^{G\times\C^*}(F_1).\] 

By the Thom isomorphism,  $H_*^{G\times\C^*}(Z_\mathscr{O})\simeq H_*^{G\times\C^*}(\mathscr{O})$. Therefore, any $\gamma\in H_*^{G\times\C^*}(\mathscr{O})$ defines an $H^{G\times\C^*}_*(\pt)$-modules homomorphism $\rho_{\gamma}:H_*^{G\times\C^*}(F_2)\rightarrow H_*^{G\times \C^*}(F_1)$, and we need the following useful formula. 
\begin{lem}\cite[Lemme 2]{V93}\label{Lemme 2}
    For any $\gamma\in H_*^{G\times\C^*}(\mathscr{O})$ and $\gamma_2\in H_*^{G\times \C^*}(F_2)$,
    \[\rho_\gamma(\gamma_2)= p_{1,\mathscr{O}*}\bigg(\Eu(q^{-2}\otimes T^*_{p_{1,\mathscr{O}}})\cdot p_{2,\mathscr{O}}^*\gamma_2\cap \gamma\bigg)\in H_*^{G\times \C^*}(F_1),\]
    where $T^*_{p_{1,\mathscr{O}}}$ is the relative cotangent bundle, and $\Eu(q^{-2}\otimes T^*_{p_{1,\mathscr{O}}})$ is the $T\times\C^*$-equivariant Euler class of the tensor product of $T^*_{p_{1,\mathscr{O}}}$ with the trivial line bundle with the nontrivial $q^{-2}$ action of $\C^*$.
\end{lem}

For the computations, we will frequently use the localization formula in equivariant Borel--Moore homology. Let $T\subset G$ be a maximal torus, and let $X$ be a smooth projective variety such that the torus fixed point set $X^T$ is finite. First of all, we have $H_*^G(X)\simeq H_*^T(X)^W$, where $W$ is the Weyl group. Let $\pi:X\rightarrow \pt$ be the structure morphism. Then for any $\gamma\in H_*^T(X)$, we have the following localization formula \cite{AB84}
\begin{equation}\label{equ:local}
    \pi_*(\gamma)=\sum_{x\in X^T}\frac{\gamma|_x}{\Eu(T_xX)}\in H_*^T(\pt),
\end{equation}
where $\gamma|_x\in H_*^T(\pt)$ is the pullback of $\gamma$ to the fixed point $x\in X^T$, and $\Eu(T_xX)=\prod_{\mu_i}\mu_i\in H_*^T(\pt)$ with the product over all the torus weights $\{\mu_i\}$ in the $T$-vector space $T_xX$.

\subsection{Partial flag varieties of type C}
\label{subsec:flag}

Let $V :=\C^{2d}$ with a non-degenerate skew-symmetric bilinear form $(-,-)$ given by the matrix 
$\begin{pmatrix}
    0& J_d\\-J_d&0
\end{pmatrix},$ 
where \[J_d=\begin{pmatrix}
    & & 1\\
    & \adots &\\
    1 & &
\end{pmatrix}_{d\times d}.\]
Throughout the paper, we set
\[
G = \Sp(V) \quad\text{ and } \quad N=2n+1,
\]
for a fixed positive integer $n$. For $1\leq i\leq 2n$, let $\tau i:=N-i$. Let 
\begin{align}  \label{eq:Lamdacd}
\Lambda_{\mathfrak{c},d}= \big\{ \mathbf{v}=(v_i) \in \mathbb{N}^{N} \mid   v_i = v_{N+1-i},\quad \textstyle \sum_{i=1}^{N} v_i = 2d \big\}.
\end{align}
For any subspace $W\subseteq V$, let $W^{\perp} = \{x \in V\mid (x,y)=0, \ \forall y\in W\}$.
For any $\mathbf{v} \in \Lambda_{\mathfrak{c},d}$, define 
$$
\mathscr{F}_{\mathbf{v}}=\{ F=( 0=V_0 \subset V_1 \subset \cdots \subset V_{N}=V)\ \mid \  V_i=V_{N-i}^{\perp} ,\ \text{dim}( V_{i}/V_{i-1})= v_i,\  \forall i \}. 
$$
The natural $G$-action on $V$ induces a natural transitive action of $G$ on $\mathscr{F}_{\mathbf{v}}$, and thus 
\begin{align}
 \label{eq:F}
\mathscr{F} = \bigsqcup _{\mathbf{v} \in \Lambda_{\mathfrak{c},d}} \mathscr{F}_{\mathbf{v}} 
\end{align}
is a $G$-variety called the {\em $N$-step partial flag variety}. Let $\{\epsilon_i\mid 1\leq i\leq 2d\}$ be the standard basis of $V$. Let $F_\mathbf{v}$ be the flag in $\mathscr{F}_{\mathbf{v}}$ such that for $1\leq i\leq n$, $V_i=\Span\{\epsilon_j\mid j\in [\mathbf{v}]_i\}$ and $V_{N-i}=V_i^\perp$. Let $P_\mathbf{v}$ be the stabilizer of the flag $F_\mathbf{v}$ inside $G$, then 
\begin{equation}\label{equ:FvGP}
    G/P_{\mathbf{v}}\simeq \mathscr{F}_{\mathbf{v}}.
\end{equation}
Let 
\[\mathscr{M}:=T^*\mathscr{F}\]
be the cotangent bundle of $\mathscr{F}$.

Let $W_{\mathfrak{c}} = \mathbb{Z}_2^{d}\rtimes S_{d}$ be the Weyl group of type $C_d$, which
has a natural action on the set $\{1,2, \ldots,2d\}$. For any $\mathbf{v}=(v_1,\cdots,v_N)\in \Lambda_{\mathfrak{c},d}$, let $\bar{v}_i:=v_1+\cdots v_i$, $[\mathbf{v}]_i=[1+\bar{v}_{i-1},\bar{v}_i]$ for $1\leq i\leq N$, and $[\mathbf{v}]^{\mathfrak{c}}_{n+1}:=[1+\bar{v}_n,d]$.
Let $$[\mathbf{v}]^{\mathfrak{c}}=([\mathbf{v}]_1,\cdots,[\mathbf{v}]_n,[\mathbf{v}]^{\mathfrak{c}}_{n+1}),$$ and $$W_{[\mathbf{v}]^{\mathfrak{c}}}:=S_{[\mathbf{v}]_1}\times\cdots\times S_{[\mathbf{v}]_n}\times(\mathbb{Z}_{2}^{|[\mathbf{v}]^{\mathfrak{c}}_{n+1}|}\rtimes S_{[\mathbf{v}]_{n+1}^{\mathfrak{c}}})\subset W_{\mathfrak{c}}.$$ It is the Weyl subgroup corresponding to the parabolic subgroup $P_\mathbf{v}$. 

Denote
\begin{align}
\Xi_d=\Big\{A=(a_{ij}) \in {\rm Mat}_{N\times N}(\mathbb{N})\ | \ \sum_{i,j}a_{ij}=2d,\ a_{ij}=a_{N+1-i,N+1-j},\ \forall \  i, j \Big\}.
\end{align}
To a matrix $A \in \Xi_d$, we associate a partition of the set $\{1,2,\cdots,2d\}$ 
    \begin{equation*}
      [A]=([A]_{11}, \cdots [A]_{1N}, [A]_{21},\cdots, [A]_{NN}),
    \end{equation*}
where $[A]_{ij} =[\sum\limits_{(h,k)< (i,j)}a_{hk}+1 , \sum\limits_{(h,k)<(i,j)}a_{hk}+a_{ij}]\subseteq \mathbb{N}$,
and $<$ is the  left lexicographical order, i.e.,
$$(h,k) < (i,j) \Leftrightarrow h<i  ~\text{or}~ (h = i ~\text{and}~  k<j).$$ 
Let $[A]^{\mathfrak{c}}_{n+1,n+1}= [\sum_{(h,k)< (n+1,n+1)}a_{hk}+1, d]$.
Define $[A]^{\mathfrak{c}}$ to be the following partition of the set $\{1,2,\cdots,d\}$,
\[[A]^{\mathfrak{c}} = ([A]_{11}, \cdots [A]_{1N}, [A]_{21},\cdots, [A]_{n+1,n}, [A]_{n+1,n+1}^{\mathfrak{c}}).
\]
and define subgroups of the Weyl group by 
\[W_{[A]^{\mathfrak{c}}}=S_{[A]_{11}}\times\dotsb \times S_{[A]_{1,N}}\times S_{[A]_{21}} \times\dotsb\times S_{[A]_{n+1,n}}\times(\mathbb{Z}_2^{|[A]_{n+1,n+1}^{\mathfrak{c}}|}\rtimes S_{[A]_{n+1,n+1}^{\mathfrak{c}}}).\]

For any matrix $A \in \Xi_d$, denote
\begin{equation*}
  \ro(A) = (\sum_j a_{ij})_{i=1, 2,\cdots, N}\in\Lambda_\mathfrak{c},\quad\ {\rm and}\  \co(A) = (\sum_i a_{ij})_{j=1,2, \cdots, N}\in\Lambda_\mathfrak{c}.
\end{equation*}
For any $\mathbf{v}, \mathbf{w} \in \Lambda_{\mathfrak{c}}$, let
$$\Xi_d(\mathbf{v},\mathbf{w}) = \{ A \in \Xi_d \mid \ro(M) = \mathbf{v},\  \co(A) = \mathbf{w} \}.$$

For a pair of flags $(F,F')\in \mathscr{F}_{\mathbf{v}}\times \mathscr{F}_{\mathbf{w}}$, define an $N\times N$ matrix $A=(a_{i,j})$ by setting
\begin{equation}\label{equ:ObirtA}
    a_{i,j}=\dim \frac{V_i\cap V_j'}{V_{i-1}\cap V_j'+V_{i}\cap V_{j-1}'}.
\end{equation}
Then it is easy to show that 
\begin{equation}\label{equ:ViV''j}
    \dim(V_i\cap V'_j)=\sum\limits_{s\le i,t\le j}a_{st}.
\end{equation}
It has been shown in \cite[Section 6]{BKLW18} that this gives a bijection between the diagonal $G$-orbits in $\mathscr{F}_{\mathbf{v}}\times \mathscr{F}_{\mathbf{w}}$ and $\Xi_d(\mathbf{v},\mathbf{w})$. For any $A\in \Xi_d$, let $\mathscr{O}_A$ denote the corresponding $G$-orbit on $\mathscr{F}\times \mathscr{F}$, and let $Z_A$ denote its conormal bundle in $\mathscr{M}\times \mathscr{M}$. In particular, if $A=\diag(\mathbf{v})$ for some $\mathbf{v}\in \Xi_{\mathfrak{c},d}$, then $Z_{\diag(\mathbf{v})}$ is the diagonal copy of $T^*\mathscr{F}_\mathbf{v}$ inside $Z$. On the other hand, it is well known that the diagonal $G$-orbits on $\mathscr{F}_{\mathbf{v}}\times \mathscr{F}_{\mathbf{w}}$ are in bijection with the double cosets $W_{[\mathbf{v}]^{\mathfrak{c}}}\backslash W/W_{[\mathbf{w}]^{\mathfrak{c}}}$. 
 
We can define an order $\preceq$ on $\Xi_d$ as follows.
For any $A=(a_{ij}), B=(b_{ij})\in \Xi_d$, $A \preceq B $ if and only if
\begin{align}\label{equ:order}
\ro(A) = \ro(B),\ \co(A)=\co(B),\ {\rm and}\ \sum_{r\leq i; s\geq j} a_{rs} \leq \sum_{r\leq i; s\geq j} b_{rs}, \  \forall  i<j.
\end{align}
This order is compatible with the Bruhat order on $W_{\mathfrak{c}}$ via the above bijection.
\begin{prop}\cite{BKLW18}\label{prop:orbit}
For any $A,B \in \Xi_d$, $A \preceq B$ if $\mathcal{O}_A \subseteq  \overline{\mathcal{O}}_B$.
\end{prop}

For $A,B\in\Xi_d$ such that $\co(A)=\ro(B)$, the set $p_{12}^{-1}(\mathcal{O}_A)\cap p_{23}^{-1}(\mathcal{O}_B)$ is $G$-stable. Let $$\mathbf{M}(A,B):=\{C\in\Xi_d\mid \mathcal{O}_C\subset p_{13}(p_{12}^{-1}(\mathcal{O}_A)\cap p_{23}^{-1}(\mathcal{O}_B))\}.$$
By the same argument as in \cite[Proposition 5(c)]{V98}, we get
\begin{prop}\label{p-fiber}
There exists a unique element $A\circ B\in \mathbf{M}(A,B)$, such that $\mathbf{M}(A,B)\subset \{C\mid C\preceq A\circ B\}$. 
\end{prop}

Let $E_{ij}$ be the standard $N\times N$ matrix unit with 1 at $(i,j)$-entry. For $\mathbf{v}\in \Lambda_{\mathfrak{c},d-a}$,  define 
\begin{align} \label{eq:Eijv}
E_{ij}^{\theta} :=  E_{ij} + E_{\tau i +1,\tau j+1}, \qquad
E_{ij}^{\theta}(\mathbf{v},a):= \diag(\mathbf{v}) +  aE_{ij}^{\theta}.
\end{align} 
Then $E_{ij}^{\theta}=E_{\tau i+1,\tau j+1}^{\theta}$. Let $\mathbf{e}_i$ ($1\leq i\leq N$) be the standard basis for $\C^N$ (viewed as row vectors). From definition, we get
\begin{align}\label{equ:OA}
		\mathcal{O}_{E_{i,i+1}^{\theta}(\mathbf{v},a)}=&\bigg\{(F,F')\mid 
		F=(V_k)_{0\le k \le N} 
			, F'=(V'_k)_{0\le k \le N},\\
			& V'_i\stackrel{a}{\subset} V_i, V_k=V'_k \textit{ if } k\neq i,\tau i\bigg\}\subset \mathscr{F}_{\mathbf{v}+a\mathbf{e}_i+a\mathbf{e}_{\tau i+1} }\times \mathscr{F}_{\mathbf{v}+a\mathbf{e}_{i+1}+a\mathbf{e}_{\tau i} },\notag
\end{align}
and it is a closed orbit. Here $V'_i\stackrel{a}{\subset} V_i$ means that $V_i'$ is a vector subspace in $V_i$ of codimension $a$.

\begin{lem}\label{lem:MAB}
    Fix $a\in\mathbb{Z}_{\geq 0}$,  $\mathbf{v}\in\Lambda_{\mathfrak{c},d-a}$, and let $A=E_{h,h+1}^{\theta}(\mathbf{v},a)$.
    For $B\in\Xi_d$ with $\co(A)=\ro(B)$, if $h\neq n$,
    \[\mathbf{M}(A,B)=\bigg\{B+\sum_{1\leq j\leq N} s_j(E^\theta_{hj}-E^\theta_{h+1,j})\mid s_j\in \mathbb{Z}_{\geq 0}, s_j\leq b_{h+1,j},\sum_j s_j=a\bigg\},\]
    and if $h=n$,
    \[\mathbf{M}(A,B)=\bigg\{B+\sum_{1\leq j\leq N} s_j(E^\theta_{nj}-E^\theta_{n+1,j})\mid s_j\in \mathbb{Z}_{\geq 0}, s_j+s_{\tau j+1}\leq b_{n+1,j},\sum_j s_j=a\bigg\}.\]
\end{lem}
\begin{proof}
    Pick $(F,F',F'')\in p_{12}^{-1}(\mathcal{O}_A)\cap p_{23}^{-1}(\mathcal{O}_B)$, and assume $(F,F'')\in \mathcal{O}_C$ for some matrix $C$. Then by \eqref{equ:ObirtA} and the above formula \eqref{equ:OA} for $\mathcal{O}_A$, $c_{ij}=b_{ij}$ if $i\neq h, h+1, \tau h, \tau h+1$. 
    
    Let us first assume $h\neq n,n+1$, then $\tau h\neq h+1$ and $\tau h\neq h-1$.
    Suppose 
    \[\dim V_{h}\cap V_j'' = \dim V'_{h}\cap V_j''+\sum_{1\leq i\leq j}s_i\]
    for some $s_i\in \mathbb{Z}_{\geq 0}$. Then $\sum_i s_i=a$. Using the fact that for any two linear subspaces $U_1,U_2\subset V$, \[\dim (U_1^\perp\cap U_2^\perp)=\dim (U_1+U_2)^\perp=2d-\dim U_1-\dim U_2+\dim U_1\cap U_2,\] we get 
    \begin{align*}
    \dim V_{\tau h}\cap V_{\tau j}'' 
    =&2d-\dim V_h-\dim V_{j}''+\dim V_{h}\cap V_j''\\
    =&\dim V'_{\tau h}\cap V_{\tau j}''+\sum_{1\leq i\leq j}s_i-a.
    \end{align*} 
    Therefore,
    \begin{align*}
    	c_{h+1,j}=&\dim V_{h+1}\cap V_j''/(V_{h}\cap V_j''+V_{h+1}\cap V_{j-1}'')\\
    	=&\dim V'_{h+1}\cap V_j''-\dim V_{h}\cap V_j''-\dim V'_{h+1}\cap V_{j-1}''+\dim  V_{h}\cap V_{j-1}''\\
    	=&b_{h+1,j}-s_j,
    \end{align*}
    \begin{align*}
    	c_{h,j}=&\dim V_{h}\cap V_j''/(V_{h-1}\cap V_j''+V_{h}\cap V_{j-1}'')\\
    	=&\dim V_{h}\cap V_j''-\dim V'_{h-1}\cap V_j''-\dim V_{h}\cap V_{j-1}''+\dim  V'_{h-1}\cap V_{j-1}''\\
    	=&b_{h,j}+s_j,
    \end{align*}
    \begin{align*}
    	c_{\tau h+1,\tau j+1}=&\dim V_{\tau h+1}\cap V_{\tau j+1}''/(V_{\tau h}\cap V_{\tau j+1}''+V_{\tau h+1}\cap V_{\tau j}'')\\
    	=&\dim V'_{\tau h+1}\cap V_{\tau j+1}''-\dim V_{\tau h}\cap V_{\tau j+1}''-\dim V'_{\tau h+1}\cap V_{\tau j}''+\dim V_{\tau h}\cap V_{\tau j}''\\
    	=&b_{\tau h+1,\tau j+1}+s_j,
    \end{align*}
    and
    \begin{align*}
    	c_{\tau h,\tau j+1}=&\dim V_{\tau h}\cap V_{\tau j+1}''/(V_{\tau h-1}\cap V_{\tau j+1}''+V_{\tau h}\cap V_{\tau j}'')\\
    	=&\dim V_{\tau h}\cap V_{\tau j+1}''-\dim V'_{\tau h-1}\cap V_{\tau j+1}''-\dim V_{\tau h}\cap V_{\tau j}''+\dim V'_{\tau h-1}\cap V_{\tau j}''\\
    	=&b_{\tau h,\tau j+1}-s_j.
    \end{align*}
    Thus, $C=B+\sum_{1\leq j\leq N} s_j(E^\theta_{hj}-E^\theta_{h+1,j})$. 
    
    Now let us assume $h=n$. Suppose 
    \[\dim V_{n}\cap V_j'' = \dim V'_{n}\cap V_j''+\sum_{1\leq i\leq j}s_i\]
    for some $s_i\in \mathbb{Z}_{\geq 0}$. Then $\sum_i s_i=a$, and  
    \begin{align*}
    	\dim V_{n+1}\cap V_{\tau j}'' 
    	=&\dim V'_{n+1}\cap V_{\tau j}''+\sum_{1\leq i\leq j}s_i-a.
    \end{align*} 
    Therefore,
    \begin{align*}
    	c_{n,j}=&\dim V_{n}\cap V_j''/(V_{n-1}\cap V_j''+V_{n}\cap V_{j-1}'')\\
    	=&\dim V_{n}\cap V_j''-\dim V'_{n-1}\cap V_j''-\dim V_{n}\cap V_{j-1}''+\dim  V'_{n-1}\cap V_{j-1}''\\
    	=&b_{n,j}+s_j,
    \end{align*}
    \begin{align*}
    	c_{n+2,\tau j+1}=&\dim V_{n+2}\cap V_{\tau j+1}''/(V_{n+1}\cap V_{\tau j+1}''+V_{n+2}\cap V_{\tau j}'')\\
    	=&\dim V'_{n+2}\cap V_{\tau j+1}''-\dim V_{n+1}\cap V_{\tau j+1}''-\dim V'_{n+2}\cap V_{\tau j}''+\dim V_{n+1}\cap V_{\tau j}''\\
    	=&b_{n+2,\tau j+1}+s_j,
    \end{align*}
    and 
    \begin{align*}
    	c_{n+1,j}=&\dim V_{n+1}\cap V_j''/(V_{n}\cap V_j''+V_{n+1}\cap V_{j-1}'')\\
    	=&\dim V_{n+1}\cap V_j''-\dim V_{n}\cap V_j''-\dim V_{n+1}\cap V_{j-1}''+\dim  V_{n}\cap V_{j-1}''\\
    	=&b_{n+1,j}-s_j-s_{\tau j+1}\geq 0.
    \end{align*}
    Thus, \begin{align*}
        C&=B+\sum_{1\leq j\leq N} s_jE^\theta_{nj}-\sum_{1\leq j\leq N}s_jE_{n+1,j}-\sum_{1\leq j\leq N}s_{\tau j+1}E_{n+1,j}\\&=B+\sum_{1\leq j\leq N} s_j(E^\theta_{nj}-E^\theta_{n+1,j}).
    \end{align*} 
    
    Finally, let us deal with the case $h=n+1$. As before, suppose 
    \[\dim V_{n+1}\cap V_j'' = \dim V'_{n+1}\cap V_j''+\sum_{1\leq i\leq j}s_i\]
    for some $s_i\in \mathbb{Z}_{\geq 0}$. Then $\sum_i s_i=a$, and  
    \begin{align*}
    	\dim V_{n}\cap V_{\tau j}'' 
    	=&\dim V'_{n}\cap V_{\tau j}''+\sum_{1\leq i\leq j}s_i-a.
    \end{align*} 
    Therefore,
     \begin{align*}
    	c_{n,\tau j+1}=&\dim V_{n}\cap V_{\tau j+1}''/(V_{n-1}\cap V_{\tau j +1}''+V_{n}\cap V_{\tau j}'')\\
    	=&\dim V_{n}\cap V_{\tau j+1}''-\dim V'_{n-1}\cap V_{\tau j+1}''-\dim V_{n}\cap V_{\tau j}''+\dim  V'_{n-1}\cap V_{\tau j}''\\
    	=&b_{n,\tau j}-s_{j},
    \end{align*}
    \begin{align*}
    	c_{n+2,j}=&\dim V_{n+2}\cap V_{j}''/(V_{n+1}\cap V_{j}''+V_{n+2}\cap V_{j-1}'')\\
    	=&\dim V'_{n+2}\cap V_{j}''-\dim V_{n+1}\cap V_{j}''-\dim V'_{n+2}\cap V_{j-1}''+\dim V_{n+1}\cap V_{j-1}''\\
    	=&b_{n+2,j}-s_j,
    \end{align*}
    and 
    \begin{align*}
    	c_{n+1,j}=&\dim V_{n+1}\cap V_j''/(V_{n}\cap V_j''+V_{n+1}\cap V_{j-1}'')\\
    	=&\dim V_{n+1}\cap V_j''-\dim V_{n}\cap V_j''-\dim V_{n+1}\cap V_{j-1}''+\dim  V_{n}\cap V_{j-1}''\\
    	=&b_{n+1,j}+s_j+s_{\tau j+1}.
    \end{align*}
    Thus, \begin{align*}
        C=&B-\sum_{1\leq j\leq N} s_j(E_{n,\tau j+1}+E_{n+2,j})+\sum_{1\leq j\leq N}(s_j+s_{\tau j+1})E_{n+1,j}\\
        =&B+\sum_{1\leq j\leq N} s_j(E^\theta_{n+1,j}-E^\theta_{n+2,j}).
    \end{align*} This finishes the proof of the lemma. 
\end{proof}

Let $C\in \Xi_d(\mathbf{v},\mathbf{w})$ be a non-diagonal matrix, and let
\[(h,l):=\max\{(i,j)\mid 1\leq i< j\leq N, c_{ij}\neq 0\},\]
with respect to the right-lexicographic order, i.e., $(i,j)<(k,l)$ if $j<l$ or $j=l$ and $i<k$. Let 
\[B=C+c_{hl}(E_{h+1,l}^\theta-E_{hl}^\theta),\quad \text{and}\quad A=E^\theta_{h,h+1}(\mathbf{v}-c_{hl}\mathbf{e}_h-c_{hl}\mathbf{e}_{N+1-h},c_{hl}).\] Then $\co(A)=\ro(B)$ since for any $j$,
\begin{align*}
    \sum_{i}a_{ij} &= a_{jj}+\delta_{j,h+1}c_{h,l}+\delta_{j,N-h}c_{hl}\\   
    &= v_j-\delta_{j,h}c_{hl}-\delta_{j,N+1-h}c_{hl}+\delta_{j,h+1}c_{h,l}+\delta_{j,N-h}c_{hl}\\
    &= \sum_i c_{ji}-\delta_{j,h}c_{hl}-\delta_{j,N+1-h}c_{hl}+\delta_{j,h+1}c_{h,l}+\delta_{j,N-h}c_{hl}
    =\sum_i b_{ji}.
\end{align*}
\begin{cor}\label{cor:AcircB}
     Let $A,B,C$ be as above, and let $\mathcal{O}_T:=p_{12}^{-1}(\mathcal{O}_A)\cap p_{23}^{-1}(\mathcal{O}_B)\cap p_{13}^{-1}(\mathcal{O}_C)$. Then
     \begin{enumerate}
         \item $A\circ B=C$;
         \item  the projection $p_{13}|_{\mathcal{O}_T}:\mathcal{O}_T\rightarrow \mathcal{O}_C$ is an ismorphism.
     \end{enumerate}
\end{cor}
\begin{proof}
    The first statement follows from the above Lemma \ref{lem:MAB} and the partial order \eqref{equ:order}. The matrix $C$ corresponds to the case $s_j=\delta_{j,l}c_{hl}$. Notice that if $(h,l)=(n,n+1)$, then $b_{n+1,n+1}=c_{n+1,n+1}+2c_{n,n+1}\geq 2s_{n+1}$.

    Let us prove the second one. Let $(F,F',F'')\in \mathcal{O}$. Then it is enough to show that $F'$ is uniquely determined by $(F,F'')$. 
    Since $(F,F')\in \mathcal{O}_A$, we only need to show that $V_h'$ is uniquely determined by $(F,F'')$.  
    Recall the matrix $C$ corresponds to the choice of $s_j=\delta_{jl}c_{hl}$. Hence, 
    \[ V'_h\stackrel{c_{hl}}{\subset} V_h, \textit{ and } \dim V_{h}\cap V_j'' = \dim V'_{h}\cap V_j''+c_{hl}\delta_{j\geq l}.\]
    Notice that $\dim V_h=v_1+\cdots +v_h$, and  $\dim V'_h=v_1+\cdots +v_h-c_{hl}$. Recall \[(h,l)=\max\{(i,j)\mid 1\leq i<j\leq n, c_{ij}\neq 0\},\] we get by \eqref{equ:ViV''j}
	\begin{equation*}\label{eq}
	  \dim V_h\cap V_l''=v_1+\cdots +v_h=\dim V_h.  
	\end{equation*}
	Thus, $V_h\cap V_l''=V_h$. On the other hand,
	\[\dim V'_h\cap V_l''=\dim V_h\cap V_l''-c_{hl}=\dim V_h-c_{hl}=\dim V'_h.\]
	We get $V'_h\cap V_l''=V'_h$. Observe that $b_{hl}=0$, we get
	\[V'_h=V'_h\cap V''_l=V_{h-1}\cap V''_l+V'_h\cap V''_{l-1}=V_{h-1}\cap V''_l+V_h\cap V''_{l-1}.\]
	Hence, $V_h'$ is uniquely determined by $(F,F'')$.
\end{proof}

\subsection{A pushforward formula} \label{subsec:push}

Let 
\[
\mathbf{R}: = \mathbb{C}[x_1,x_2,\cdots,x_d]\simeq H_*^G(G/B), 
\]
where $B$ is a Borel subgroup of $G$. We shall define a natural action of the Weyl group $W_{\mathfrak{c}}$ on $\mathbf{R}$ below, which by restriction leads to actions of the subgroups $W_{[\mathbf{v}]^{\mathfrak{c}}}$ and $W_{[A]^{\mathfrak{c}}}$ on $\mathbf{R}$. The action of $\sigma \in S_{d}$ on $\mathbf{R}$ is given by
$$\sigma: \mathbf{R} \longrightarrow \mathbf{R}, 
\qquad
f(x_1,x_2,\cdots,x_d) \mapsto  f(x_{\sigma(1)},x_{\sigma(2)},\cdots,x_{\sigma(d)}).$$
For any $m\in [1,d]$, the generator $\iota_m$ in the m-th copy of $\mathbb{Z}_2$ in $\mathbb{Z}_2^d$ acts on $\mathbf{R}$ by
\begin{equation*}
  \begin{split}
\iota_m: \mathbf{R} &\longrightarrow \mathbf{R},\\
 f(x_1,\cdots,x_{m-1}, x_m,x_{m+1},\cdots x_d) &\mapsto  f(x_1,\cdots,x_{m-1}, -x_m,x_{m+1},\cdots x_d).
  \end{split}
\end{equation*}
For two subgroups of the Weyl group $W_1\subset W_2\subset W_\mathfrak{c}$, we define a map
 $$W_2/W_1 : \mathbf{R}^{W_1} \longrightarrow \mathbf{R}^{W_2}, \qquad f \mapsto \sum\limits_{\sigma \in W_2/W_1}\sigma(f).$$

To simplify the notations, we shall denote $\mathbf{R}^{W_{[\mathbf{v}]^{\mathfrak{c}}}}$ by $\mathbf{R}^{[\mathbf{v}]^{\mathfrak{c}}}$, 
and $\mathbf{R}^{W_{[A]^{\mathfrak{c}}}}$ by $\mathbf{R}^{[A]^{\mathfrak{c}}}$. 

\begin{prop}\label{prop:pushforward}
Let $\mathbf{v}, \mathbf{v}_1, \mathbf{v}_2 \in \Lambda_{\mathfrak{c},d}$, and $A \in \Xi_d(\mathbf{v}_1,\mathbf{v}_2)$.
\begin{enumerate}
    \item There exist isomorphisms $H_*^{G}(\mathscr{F}_{\mathbf{v}}) \cong \mathbf{R}^{[\mathbf{v}]^{\mathfrak{c}}} \text{and} \ H_*^{G}(\mathcal{O}_A) \cong \mathbf{R}^{[A]^{\mathfrak{c}}}$.
    \item Both the projection maps $p_{i,A} : \mathcal{O}_A \rightarrow \mathscr{F}_{\mathbf{v}_i}$ are smooth fibrations. Moreover, if $\mathcal{O}_A $ is closed, then the direct image morphism $p_{{1,A} *}$ is given by
    $$p_{{1,A} *}(\gamma) =  W_{[\mathbf{v}_1]^{\mathfrak{c}}}/W_{[A]^{\mathfrak{c}}}\bigg(\frac{\gamma}{\Eu(T_{p_{1,A}})}\bigg),$$ 
    where $\gamma\in H_*^{G}(\mathcal{O}_A)$ and $T_{p_{1,A}}$ is the relative tangent bundle.
\end{enumerate}
\end{prop}
\begin{proof}
By the isomorphism \eqref{equ:FvGP} and the reduction in equivariant cohomology,
\[
H_*^{G}(\mathscr{F}_{\mathbf{v}}) \simeq H_*^G(G/P_\mathbf{v})\simeq  H^*_{P_\mathbf{v}}(\pt)\simeq \mathbf{R}^{[\mathbf{v}]^{\mathfrak{c}}},
\]
and the isomorphism is given by 
\[
H_*^{G}(\mathscr{F}_{\mathbf{v}})\ni \gamma \mapsto \gamma|_{F_\mathbf{v}}\in \mathbf{R}^{[\mathbf{v}]^\mathfrak{c}},
\]
where $\gamma|_{F_\mathbf{v}}$ denotes the restriction of $\gamma$ to the torus fixed point $F_\mathbf{v}\in G/{P_\mathbf{v}}$.

Given $A \in \Xi_d(\mathbf{v}_1,\mathbf{v}_2)$, define a decomposition $V=\bigoplus_{1\leq i,j\leq N}V_{ij}$ by
\[V_{ij}=\Span\{\epsilon_k\mid k\in [A]_{ij}\}.\]
Let $F_\circ =(V_k)_{1\leq k\leq N}$ be the flag with $V_k=\bigoplus\limits_{1\leq i\leq k,1\leq j\leq N}V_{ij}$, and let $F'_\circ =(V'_k)_{1\leq k\leq N}$ be the flag with $V'_k=\bigoplus\limits_{1\leq j\leq k,1\leq i\leq N}V_{ij}$. 
Then $F_\circ \in \mathscr{F}_{\ro(A)}$, $F'_\circ \in\mathscr{F}_{\co(A)}$ and $(F_\circ,F'_\circ)\in \mathscr{O}_A$. Let $P_{F_\circ}$ (respectively, $P_{F'_\circ}$) be the stabilizer of $F_\circ$ (respectively, $F'_\circ$). We have $\mathscr{O}_A\simeq G/(P_{F_\circ}\cap P_{F'_\circ})$. The reductive part of $P_{F_\circ}\cap P_{F'_\circ}$ is isomorphic to 
\[
GL(V_{11})\times \dotsb\times GL(V_{N1})\times GL(V_{21})\times \dotsb \times GL(V_{n+1,n})\times \Sp(V_{n+1,n+1}),
\]
with Weyl group $W_{[A]^\mathfrak{c}}$. Thus we have an isomorphism
\[ H_*^{G}(\mathcal{O}_A) \cong \mathbf{R}^{[A]^{\mathfrak{c}}},\]
given by
$\gamma\mapsto \gamma|_{(F_\circ,F'_\circ)}$.
This finishes the proof of Part (1). By definition, $P_{F_\circ}=P_{\mathbf{v}_1}$. Hence, the projection $p_{1,A}:\mathcal{O}_A\simeq G/(P_{F_\circ}\cap P_{F'_\circ})\rightarrow \mathscr{F}_{\mathbf{v}_1}\simeq G/P_{F_\circ}$ is a smooth fibration. The same holds for the other projection. By the construction, the torus fixed points in $(\mathcal{O}_A)^T\cap p_{1,A}^{-1}(F_{\mathbf{v}_1})$ are in bijection with the coset $W_{[\mathbf{v}_1]^{\mathfrak{c}}}/W_{[A]^{\mathfrak{c}}}$, where $T\subset G$ is the maximal diagonal torus. Then the formula for the projection $p_{1,A*}$ follows from this fact and the localization formula \eqref{equ:local}. 
\end{proof}

\subsection{The Steinberg variety}
\label{steinberg variety of type c}

Let $\mathscr{M} = T^*\mathscr{F} $ be the cotangent bundle of the $N$-step partial flag variety $\mathscr{F}$ in \eqref{eq:F}.
More explicitly, $\mathscr{M}$ can be written as
\begin{equation*}
  \label{tangent}
  \mathscr{M} = T^*\mathscr{F} = \{(F,x) \in \mathscr{F} \times \mathfrak{sp}_{2d}\mid x(F_i) \subseteq F_{i-1}, ~ \forall i \}
  \subseteq \mathscr{F} \times \mathfrak{sp}_{2d},
\end{equation*}
where $\mathfrak{sp}_{2d}$ is the Lie algebra of $G$. There is a natural $G$-action on $\mathscr{M}$ induced by the $G$-action on $\mathscr{F}$. Define a $G\times \C^*$-action on $\mathscr{M}$ by
$$ (g,z) \cdot  (F,x) = (gF, z^{-2}gxg^{-1}), \quad \forall (g,z) \in G\times \C^*.$$ As before, let $q$ denote the degree one representation of $\mathbb{C}^*$ and $\hbar:=c_1(q^{-2})$. 

Let $\mathscr{N}$ be the nilpotent variety of the Lie algebra $\mathfrak{sp}_{2d}$. By definition, we have $\mathscr{M} \subseteq \mathscr{F} \times \mathscr{N}$. The projection map $\pi: \mathscr{M}\rightarrow \mathscr{N}, (F,x) \mapsto x$, is proper and $G\times \C^*$-equivariant.
Let 
\begin{align}
    \label{eq:Z}
    Z:=\mathscr{M}\times_{\mathscr{N}} \mathscr{M}\subset \mathscr{M}\times\mathscr{M}
\end{align}
be the (generalized) Steinberg variety of type C. The group $G\times \C^* $ acts on $Z$ diagonally, and $(H_*^{G\times \C^*}(Z), \star)$ is a convolution $H_*^{G\times \C^*}(\pt)$-algebra with unit, see Section \ref{subsec:convolution}.

Via convolution, the algebra $H_*^{G\times \C^*}(Z)$ acts on 
\begin{align}  \label{KMR}
    H_*^{G\times \C^*}(\mathscr{M})\simeq \bigoplus_{\mathbf{v}\in \Lambda_{\mathfrak{c},d}} \mathbf{R}^{[\mathbf{v}]^{\mathfrak{c}}}[\hbar].
\end{align}
Following \cite[Claim 7.6.7]{CG97}, we have the following lemma.

\begin{lem}\label{lemma:faithful}
The convolution action of $H_*^{G\times \C^*}(Z)$ on $H_*^{G\times \C^*}(\mathscr{F})$ is faithful.
\end{lem}
\begin{rem}
    More precisely, the \textit{loc. cit.} showed that the action of $H_*^{G\times \C^*}(Z)$ on $H_*^{G\times \C^*}(\mathscr{M})$ is faithful. By the Thom isomorphism, $H_*^{G\times \C^*}(\mathscr{M})\simeq H_*^{G\times \C^*}(\mathscr{F})\simeq H_*^{G}(\mathscr{F})[\hbar]$, and it is a free $\mathbb{C}[\hbar]$ module. Let $\iota:\mathscr{F}\hookrightarrow\mathscr{M}$ denote the natural inclusion. Then after localization with respect to the $\mathbb{C}^*$ action, any class $\gamma\in H_*^{G\times \C^*}(\mathscr{M})$ can be written as 
    \[\gamma=\iota_*\bigg(\frac{\iota^*\gamma}{\Eu(T^*\mathscr{M})}\bigg).\]
    Therefore, $H_*^{G\times \C^*}(\mathscr{F})$ must also be a faithful module.
\end{rem}

\subsection{A generating set for the convolution algebra}
In this section, we establish a generating set for the convolution algebra $H^{G\times\mathbb{C}^*}_*(Z)$.
Let $Z_{\mathbf{v},\mathbf{w}}:=T^*\mathscr{F}_{\mathbf{v}}\times_{\mathscr{N}}T^*\mathscr{F}_{\mathbf{w}} $. The diagonal $G$-orbits on $\mathscr{F}_{\mathbf{v}}\times \mathscr{F}_{\mathbf{w}}$ are indexed by $\Xi_d(\mathbf{v},\mathbf{w})$, and 
\[Z_{\mathbf{v},\mathbf{w}}=\bigsqcup_{A\in \Xi_d(\mathbf{v},\mathbf{w})}Z_A,\]
where $Z_A:=T^*_{\mathcal{O}_A}(\mathscr{F}_{\mathbf{v}}\times \mathscr{F}_{\mathbf{w}})$ is the conormal bundle of the $G$-orbit corresponding to $A\in \Xi_d(\mathbf{v},\mathbf{w})$. Then all the irreducible components of $Z$ are of the form $\overline{Z_A}$. For any $A\in\Xi_d$, let $Z_{\preceq A}:=\bigcup_{B\preceq A}Z_B$, which is a closed subvariety of $Z$. By the cellular fibration lemma \cite[Lemma 5.5.1]{CG97}, the open immersion $Z_{\prec A}\hookrightarrow Z_{\preceq A}$ gives rise to the following short exact sequence \begin{equation}\label{cfb}
    0\rightarrow H^{G\times\mathbb{C}^*}_*(Z_{\prec A})\rightarrow H^{G\times\mathbb{C}^*}_*(Z_{\preceq A})\rightarrow H^{G\times\mathbb{C}^*}_*(Z_A)\rightarrow0,
\end{equation} where $Z_{\prec A}:=Z_{\preceq A}\backslash Z_A$. Moreover, all the odd homologies vanish, see \cite{DR08}. 
Hence, the induced maps $H^{G\times\mathbb{C}^*}_*(Z_{\preceq A})\rightarrow H^{G\times\mathbb{C}^*}_*(Z)$ are injective and their images form a filtration of $H^{G\times\mathbb{C}^*}_*(Z)$ indexed by $\Xi_d$. Moreover, Proposition \ref{p-fiber} implies $Z_{\preceq A}\circ Z_{\preceq A}\subset Z_{\preceq A\circ B}$. Thus, $H^{G\times\mathbb{C}^*}_*(Z_{\preceq A})\star H^{G\times\mathbb{C}^*}_*(Z_{\preceq B})\subset H^{G\times\mathbb{C}^*}_*(Z_{\preceq A\circ B})$. The following maps \begin{align*}
    H^{G\times\mathbb{C}^*}_*(Z_{\preceq A})\twoheadrightarrow H^{G\times\mathbb{C^*}}_*(Z_{\preceq A})/H^{G\times\mathbb{C}^*}_*(Z_{\prec A})&\stackrel{\sim}{\rightarrow}H^{G\times\mathbb{C}^*}_*(Z_{A})\\&\stackrel{\sim}{\rightarrow}H^{G\times\mathbb{C}^*}_*(\mathcal{O}_{A})\stackrel{\sim}{\rightarrow}\mathbf{R}^{[A]^\mathfrak{c}}[\hbar]
\end{align*}
identifies the associated graded of $H^{G\times\mathbb{C}^*}_*(Z)$ with $\bigoplus_{A\in\Xi_d}\mathbf{R}^{[A]^\mathfrak{c}}[\hbar]$. 

\begin{lem}\label{lem:Ea1}
    For $1\le i\le N-1$ and $a\in\mathbb{Z}_{\ge1}$, the image of $H^{G\times\mathbb{C}^*}_*(Z_{E^\theta_{i,i+1}(\mathbf{v},a)})$ in $(H^{G\times\mathbb{C}^*}_*(Z),\star)$ can be obtained via convolution from images of $H^{G\times\mathbb{C}^*}_*(Z_{E^\theta_{i,i+1}(\mathbf{v}',k)})$ for various $\mathbf{v}'$ and $0\le k\le a-1.$
\end{lem}
\begin{proof}
    The case of $i\neq n+1$ can be proved exactly the same as in \cite[Example on p.280]{V98}. Let us assume $i=n+1$. It suffices to show the case for $n=1$. Pick non-negative integers $a,v$, such that $a+v+1\le d$. Put 
 $$A=\begin{pmatrix}
     d-v-a-1&0&0\\
     1&2v+2a&1\\
     0&0&d-v-a-1
 \end{pmatrix},\quad B=\begin{pmatrix}
     d-v-a&0&0\\
     a&2v&a\\
     0&0&d-v-a
 \end{pmatrix},$$
 $$C=\begin{pmatrix}
     d-v-a-1&0&0\\
     a+1&2v&a+1\\
     0&0&d-v-a-1
 \end{pmatrix}.$$
From \eqref{equ:OA}, we get $A\circ B=C$. By definition, 
$$[A]^\mathfrak{c}=([1,d-v-a-1],[d-v-a,d-v-a],[d-v-a+1,d]),$$
$$[B]^\mathfrak{c}=([1,d-v-a],[d-v-a+1,d-v],[d-v+1,d]),$$
and
$$[C]^\mathfrak{c}=([1,d-v-a-1],[d-v-a,d-v],[d-v+1,d]).$$ 
Set\begin{align*}
    &I_1=[1,d-v-a-1], I_2=[d-v-a+1,d], I_3=[1,d-v-a],\\
    &I_4=[d-v-a+1,d-v], I_5=[d-v-a,d-v], I_6=[d-v+1,d].
\end{align*}
Thus,
\begin{align*}
    &W_{[A]^\mathfrak{c}}=S_{I_1}\times(\mathbb{Z}_2^{|I_2|}\rtimes S_{I_2}),W_{[B]^\mathfrak{c}}=S_{I_3}\times S_{I_4}\times(\mathbb{Z}_2^{|I_6|}\rtimes S_{I_6}), \\&W_{[C]^\mathfrak{c}}=S_{I_1}\times S_{I_5}\times(\mathbb{Z}_2^{|I_6|}\rtimes S_{I_6}).
\end{align*}
Therefore, $$W_{[C]^\mathfrak{c}}/(W_{[A]^\mathfrak{c}}\cap W_{[B]^\mathfrak{c}})=S_{I_5}/S_{I_4}.$$

Let $\mathcal{O}_T:=p^{-1}_{12}(\mathcal{O}_{A})\cap p^{-1}_{23}(\mathcal{O}_B)\cap p^{-1}_{13}(\mathcal{O}_{A\circ B})$, and let $p_{13,T}$ be the restriction of $p_{13}$ to $\mathcal{O}_T$.
By \eqref{equ:OA}, $p_{13,T}$ is a smooth and proper fibration whose fiber over $(F,F'')\in\mathcal{O}_{A\circ B}$ is 
$$\{V'_1\mid  V_1\stackrel{1}{\subset} V'_1\stackrel{a}{\subset} V_1''\}\simeq \mathrm{Gr}(1,V_1''/V_1),$$
where $F=(0=V_0\subset V_1\subset V_2\subset V_3=V)$ and $F''=(0=V''_0\subset V''_1\subset V''_2\subset V_3''=V)$. 
Let $f\in H^{G\times\mathbb{C}^*}_*(Z_A)\simeq \mathbf{R}^{[A]^\mathfrak{c}}[\hbar]$ and $g\in H^{G\times\mathbb{C}^*}_*(Z_A)\simeq \mathbf{R}^{[B]^\mathfrak{c}}[\hbar]$. By homological version of \cite[Corollary 3]{V98} and localization formula \eqref{equ:local}, we have 
$$f\star g=S_{I_5}/S_{I_4}\bigg(\frac{\mathrm{Eu}(q^{-2}T^*_{p_{13,T}})}{\mathrm{Eu}(T_{p_{13,T}})}\cdot fg\bigg),$$
where $T_{p_{13,T}}$ is the relative tangent bundle of $p_{13,T}$. Since $\mathbf{R}^{[A]^\mathfrak{c}}[\hbar]$ and $\mathbf{R}^{[B]^\mathfrak{c}}[\hbar]$ generate $\mathbf{R}^{[C]^\mathfrak{c}}[\hbar]$, it suffices to show the map 
$$\psi:\mathbf{R}^{[C]^\mathfrak{c}}[\hbar]\to\mathbf{R}^{[C]^\mathfrak{c}}[\hbar], \quad f\mapsto S_{I_5}/S_{I_4}\bigg(\frac{\mathrm{Eu}(q^{-2}T^*_{p_{13,T}})}{\mathrm{Eu}(T_{p_{13,T}})}\bigg)\cdot f$$
is surjective. Set $t=d-v-a$, then 
 $$\frac{\mathrm{Eu}(q^{-2}T^*_{p_{13,T}})}{\mathrm{Eu}(T_{p_{13,T}})}=\prod_{t+1\leq j\leq d-v}\frac{\hbar-x_j+x_t}{x_j-x_t}.$$
By the residue formula, 
\begin{align*}
    S_{I_5}/S_{I_4}\bigg(\frac{\mathrm{Eu}(q^{-2}T^*_{p_{13,T}})}{\mathrm{Eu}(T_{p_{13,T}})}\bigg)=&\sum_{t\leq k\leq d-v}(t,k)\cdot \bigg(\prod_{t+1\leq j\leq d-v}\frac{\hbar-x_j+x_t}{x_j-x_t}\bigg)\\
    =&\sum_{t\leq k\leq d-v}\prod_{\substack{t\leq j\leq d-v\\j\neq k}}\frac{\hbar-x_j+x_k}{x_j-x_k}\\
    =&\hbar^{-1}\mathrm{Res}(\prod_{t\leq j\leq d-v}\frac{\hbar-x_j+z}{x_j-z},\infty)\\
    =&(-1)^{a}(a+1),
\end{align*}
which is invertible. Hence, $\psi$ is surjective.
\end{proof}

\begin{thm}\label{thm:generators}
    The convolution algebra $H_*^{G\times\C^*}(Z)$ is generated by $H_*^{G\times\C^*}(Z_{\diag(\mathbf{v})})$ for $\mathbf{v}\in \Lambda_{\mathfrak{c},d}$, and $H_*^{G\times\C^*}(Z_{E^\theta_{i,i+1}(\mathbf{v}',1)})$ for $\mathbf{v}'\in \Lambda_{\mathfrak{c},d-1}$ and $1\leq i\leq N-1$.
\end{thm}
\begin{proof}
    For a matrix $C\in \Xi_d(\mathbf{v},\mathbf{w})$, let
    \[\ell(C):=\sum_{i<j}\begin{pmatrix}
        j-i+1\\
        2
    \end{pmatrix}c_{ij}.\]
    We will prove by induction on $\ell(C)$ that, modulo the lower graded piece $H_*^{G\times\C^*}(Z_{\prec C})$, $H_*^{G\times\C^*}(Z_{\preceq C})$ can be obtained via convolution from classes in $H_*^{G\times\C^*}(Z_A)$ such that $A\in \Xi_d$ is a diagonal matrix or a matrix of type $E^\theta_{i,i+1}(\mathbf{v}',1)$ for $1\leq i\leq N-1$. 

    If $\ell(C)=0$, then $C$ is a diagonal matrix. If $\ell(C)=1$, then $C=E^\theta_{i,i+1}(\mathbf{v}',1)$ for some $i$ and $\mathbf{v}'$. In both cases, there is nothing to show. Suppose $\ell(C)>1$. Put 
    \[(h,l):=\max\{(i,j)\mid 1\leq i< j\leq N, c_{ij}\neq 0\},\]
    with respect to the right-lexicographic order. Let 
    \[B=C+c_{hl}(E_{h+1,l}^\theta-E_{hl}^\theta),\quad \text{and}\quad A=E^\theta_{h,h+1}(\mathbf{v}-c_{hl}\mathbf{e}_h-c_{hl}\mathbf{e}_{N+1-h},c_{hl}).\]
    Thus, we are in the setting of Corollary \ref{cor:AcircB}. Hence, $A\circ B=C$. We show that the map \[\varphi:H_*^{G\times\C^*}(Z_{\preceq A})\otimes H_*^{G\times\C^*}(Z_{\preceq B})\rightarrow H_*^{G\times\C^*}(Z_{\preceq C})\twoheadrightarrow H_*^{G\times\C^*}(Z_{C})\] is surjective. The case of $(h,l)\neq (n,n+1)$ can be proved exactly the same as in \cite[Proposition 10]{V98}. Let us now assume $(h,l)= (n,n+1)$. By Corollary \ref{cor:AcircB} and the homology analog of \cite[Proposition 2]{V98}, for any $f\in \mathbf{R}^{[A]^\mathfrak{c}}$ and $g\in \mathbf{R}^{[B]^\mathfrak{c}}$, $\varphi(f\otimes g)= fg\in \mathbf{R}^{[C]^\mathfrak{c}}$. Since $b_{ij}=c_{ij}$ if $(i,j)\notin \{(n,n+1),(n+1,n+1),(n+2,n+1)\}$, $a_{n,n+1}=c_{n,n+1}$, $a_{n+1,n+1}=c_{n+1,n+1}$, and $a_{ij}=c_{ij}=0$ for $(i,j)\in \{(n,n+2),\cdots, (n,N),(n+1,1),\cdots, (n+1,n)\}$, the map $\varphi$ is obviously surjective. 
    Finally, notice that $\ell(B)<\ell(C)$. We use Lemma \ref{lem:Ea1} to conclude the proof.
\end{proof}

\section{The correspondences and operators}\label{sec:operators}

In this section, we introduce some special elements in the equivariant Borel--Moore homology of the Steinberg variety and write down the explicit formula for them under the faithful representation in Lemma \ref{lemma:faithful}.

\subsection{The combinatorial operators}\label{sec:HBoperators}
Let us introduce some notations first. Let $x_1,x_2,\cdots$, $x_{2d}$ be the standard dual coordinates on the Lie algebra of the maximal torus $T$ in $G$. For $1\leq r\leq 2d$, define $r':=2d+1-r$. Then $x_{r}=-x_{r'}$. 
Let $\varphi(x):=1+\frac{\hbar}{x}$. For any subset $I\subset[1,2d]$, define $\Phi_I(z):=\prod_{s\in I}\varphi(x_s-z)$. 
For any partition $I=(I_1,I_2,\cdots, I_N)$ of $\{1,2,\dots,2d\}$ satisfying $r\in I_j$ if and only if $r'\in I_{N+1-j}=I_{\tau j+1}$, let $(x_I)$ denote the variables 
\[(x_{i_{1,1}},\dots, x_{i_{1},j_1},x_{i_{2,1}},\dots, x_{i_{2},j_2},\dots, x_{i_{N},1},\dots,x_{i_N,j_N}),\] where $I_k=\{i_{k,1},i_{k,2},\dots,i_{k,j_k}\}$. For any $r\in I_s$, let $\tau_r^+I$ be the partition of $\{1,2,\dots,2d\}$ with $r$ shifted from $I_s$ to $I_{s+1}$ and $r'$ shifted from $I_{\tau s+1}$ to $I_{\tau s}$. For example, let us consider the case $d=3$ and $N=3$. Let $I=(\{1,2\},\{3,4\},\{5,6\})$ be a partition of $\{1,2,3,4,5,6\}$. Then \[f(x_{\tau_1^+I})=f(x_2,x_1,x_3,x_4,x_6,x_5),\]

Recall for any $\mathbf{v}\in\Lambda_{\mathfrak{c},d}$, we have the partition $[\mathbf{v}]:=([\mathbf{v}]_1,\cdots, [\mathbf{v}]_{N})$ of $\{1,2,\dots,2d\}$. We use the notation $f(x_{[\mathbf{v}]})$ to denote an element in $H_*^{G\times \C^*}(\mathscr{F}_{\mathbf{v}}) \simeq\mathbf{R}^{[\mathbf{v}]^{\mathfrak{c}}}[\hbar]$, and let 
\[\mathbf{P}:=\bigoplus_{\mathbf{v}\in \Lambda_{\mathfrak{c},d}}\mathbf{R}^{[\mathbf{v}]^\mathfrak{c}}[\hbar]\simeq H_*^{G\times \C^*}(\mathscr{F}).\]
We will define some explicit operators on $\mathbf{P}$.

For $1\leq i\leq 2n$ and $\mathbf{v}\in\Lambda_{\mathfrak{c},d}$, let
\begin{align}\label{equ:Hivu}
        H_{i,\mathbf{v}}(u):=&(1-\frac{\delta_{i,n}-\delta_{i,n+1}}{4u}\hbar)\cdot \Phi_{[\mathbf{v}]_{i}}(-u-\frac{n-i}{2}\hbar-\frac{\hbar}{4})\cdot \Phi_{[\mathbf{v}]_{\tau i}}(u+\frac{n-i}{2}\hbar+\frac{\hbar}{4}).
\end{align}
Then 
\begin{align*}
    H_{\tau i,\mathbf{v}}(-u)=
    H_{i,\mathbf{v}}(u).
\end{align*}
Consider $H_{i,\mathbf{v}}(u)$ as a series in $u^{-1}$, then it has constant $1$, and set $H^\circ_{i,\mathbf{v}}(u):=H_{i,\mathbf{v}}(u)-1$. Let $\mathbf{H}_i(u)$ be the operator on $\mathbf{P}$, which acts on $\mathbf{R}^{[\mathbf{v}]^\mathfrak{c}}$ by multiplying by $H_{i,\mathbf{v}}(u)$. For $r\geq 0$, let $H_{i,r}$ be the coefficient of $\hbar u^{-r-1}$ of $\mathbf{H}_i(u)$.

For $1\leq i\leq 2n$, define the operators $B_{i,r}$ on $\mathbf{P}$ by the following formula
\begin{equation}\label{equ:Bope}
    (B_{i,r}f)(x_{[\mathbf{v}]}):=\sqrt{-1}\sum_{j\in [\mathbf{v}]_{i}}(-x_j-\frac{n-i}{2}\hbar-\frac{\hbar}{4})^r\cdot \Phi_{[\mathbf{v}]_{i}\setminus\{j\}}(x_j)\cdot  f(x_{\tau^+_j[\mathbf{v}]}).
\end{equation}
We will construct some elements in $H_*^{G\times\C^*}(Z)$, whose convolution action on $H_*^{G\times\C^*}(\mathscr{F})\simeq\mathbf{P}$ are given by the above $H_{i,r}$ and $B_{i,r}$.

\subsection{The geometric $H$ operators}
On the flag variety $\mathscr{F}_\mathbf{v}$, we have the tautological vector bundle whose stalk at a point $F=(0=V_0\subset\cdots\subset V_N=V)\in\mathscr{F}_\mathbf{v}$ is given by $V_i$, and let $\mathcal{V}_{i,\mathbf{v}}$ denote its pullback vector bundle on $T^*\mathscr{F}_\mathbf{v}$. For $1\leq i\leq 2n$, 
let $E_{i,\mathbf{v}}$ be the virtual bundle \footnote{By taking a double cover of the torus $\C^*$, we can define $q^{\frac{1}{2}}$.}
\begin{align*}
    E_{i,\mathbf{v}}=&(q^{-\frac{5}{2}-n+i}-q^{-\frac{1}{2}-n+i})(\mathcal{V}_{i,\mathbf{v}}-\mathcal{V}_{i-1,\mathbf{v}})\\
    &+(q^{\frac{3}{2}-n+i}-q^{-\frac{1}{2}-n+i})(\mathcal{V}_{i+1,\mathbf{v}}-\mathcal{V}_{i,\mathbf{v}})+q^{\frac{\delta_{i,n}-\delta_{i,n+1}}{2}}\in K^{G\times \C^*}(T^*\mathscr{F}_\mathbf{v})
\end{align*}

Let $\mathscr{H}_{i,\mathbf{v},r}\in H^*_{G\times\mathbb{C}^*}(T^*\mathscr{F}_\mathbf{v})$ be the coefficient of $\hbar u^{-r-1}$ of the Chern polynomial $\lambda_{u^{-1}}(E_{i,\mathbf{v}})$, and
set
\begin{equation}\label{equ:geoHir}
    \mathscr{H}_{i,r}=\sum\limits_{v\in\Lambda_{\mathfrak{c},d}}\mathscr{H}_{i,\mathbf{v},r}\cdot [T^*\mathscr{F}_\mathbf{v}]\in H_*^{G\times\mathbb{C}^*}(T^*\mathscr{F}).
\end{equation}
Since the Steinberg variety $Z$ contains the diagonal copy of $T^*\mathscr{F}$, we can regard the classes $\mathscr{H}_{i,r}$ as elements in $H_*^{G\times\mathbb{C}^*}(Z)$. This will geometrize the operator $H_{i,r}$ in \eqref{equ:Hivu}.

\begin{prop}\label{prop:Hir}
    Under the isomorphism in Proposition \ref{prop:pushforward}, the convolution action of $\mathscr{H}_{i,r}\in H_{*}^{G\times \mathbb{C}^*}(Z)$ on $H_*^{G\times \C^*}(\mathscr{F})\simeq \mathbf{P}$ is given by the operator $H_{i,r}$ in Section \ref{sec:HBoperators}.
\end{prop}
\begin{proof}
  Let $f\in H^{G\times\mathbb{C}^*}_*(\mathscr{F}_\mathbf{v})\simeq\mathbf{R}^{[\mathbf{v}]^\mathfrak{c}}[\hbar]$. Since $\mathscr{H}_{i,r}$ is supported on the diagonal copy of $T^*\mathscr{F}$ inside the Steinberg variety $Z$,
  \[\mathscr{H}_{i,r}\star f=(\mathscr{H}_{i,\mathbf{v},r}\cdot[Z_{\diag(\mathbf{v})}])\star f=\mathscr{H}_{i,\mathbf{v},r}\cdot f,\]
  where $Z_{\diag(\mathbf{v})}=T^*\mathscr{F}_\mathbf{v}$ is the diagonal copy of $T^*\mathscr{F}_\mathbf{v}$ inside $Z$ and $\cdot$ is the action of equivariant cohomology on the equivariant Borel--Moore homology (see Section \ref{subsec:convolution}).

Now we calculate $\lambda_{u^{-1}}(E_{i,\mathbf{v}})$. We use the following basic fact. Recall from the proof of Proposition \ref{prop:pushforward}, $F_\mathbf{v}\in \mathscr{F}_\mathbf{v}$ is the base point. If $E$ is $G$-equivariant vector bundle on $\mathscr{F}_{\mathbf{v}}$ and the torus-weights for the fiber $E|_{F_\mathbf{v}}$ are $\{w_j\}$, then 
\[\lambda_{u^{-1}}(E)=\prod_j(1+u^{-1}w_j)\in \mathbf{R}^{[\mathbf{v}]^{\mathfrak{c}}}[u^{-1}]\cong H_*^{G}(\mathscr{F}_{\mathbf{v}})[u^{-1}].\] 

Therefore, 
\begin{align*}
    &\lambda_{u^{-1}}\bigg((q^{-\frac{5}{2}-n+i}-q^{-\frac{1}{2}-n+i})(\mathcal{V}_{i,\mathbf{v}}-\mathcal{V}_{i-1,\mathbf{v}})+(q^{\frac{3}{2}-n+i}-q^{-\frac{1}{2}-n+i})(\mathcal{V}_{i+1,\mathbf{v}}-\mathcal{V}_{i,\mathbf{v}})\bigg)\\
    =&\frac{\lambda_{u^{-1}}(q^{-\frac{5}{2}-n+i}(\mathcal{V}_{i,\mathbf{v}}-\mathcal{V}_{i-1,\mathbf{v}}))\cdot\lambda_{u^{-1}}(q^{\frac{3}{2}-n+i}(\mathcal{V}_{i+1,\mathbf{v}}-\mathcal{V}_{i,\mathbf{v}}))}{\lambda_{u^{-1}}(q^{-\frac{1}{2}-n+i}(\mathcal{V}_{i,\mathbf{v}}-\mathcal{V}_{i-1,\mathbf{v}}))\cdot\lambda_{u^{-1}}(q^{-\frac{1}{2}-n+i}(\mathcal{V}_{i+1,\mathbf{v}}-\mathcal{V}_{i,\mathbf{v}}))}
    \\
    =&\prod_{t\in [\mathbf{v}]_{i}}\frac{1+u^{-1}(x_t+\frac{n-i}{2}\hbar+\frac{5\hbar}{4})}{1+u^{-1}(x_t+\frac{n-i}{2}\hbar+\frac{\hbar}{4})}\prod_{r\in [\mathbf{v}]_{i+1}}\frac{1+u^{-1}(x_r+\frac{n-i}{2}\hbar-\frac{3\hbar}{4})}{1+u^{-1}(x_r+\frac{n-i}{2}\hbar+\frac{\hbar}{4})}\\
    =&\prod_{t\in [\mathbf{v}]_{i}}\frac{1+u^{-1}(x_t+\frac{n-i}{2}\hbar+\frac{5\hbar}{4})}{1+u^{-1}(x_t+\frac{n-i}{2}\hbar+\frac{\hbar}{4})}\prod_{s\in [\mathbf{v}]_{\tau i}}\frac{1+u^{-1}(-x_s+\frac{n-i}{2}\hbar-\frac{3\hbar}{4})}{1+u^{-1}(-x_s+\frac{n-i}{2}\hbar+\frac{\hbar}{4})}\\
    =&\Phi_{[\mathbf{v}]_{i}}(-u-\frac{n-i}{2}\hbar-\frac{\hbar}{4})\Phi_{[\mathbf{v}]_{\tau i}}(u+\frac{n-i}{2}\hbar+\frac{\hbar}{4}),
\end{align*}
where the third equality follows from the change of variable $s=r'$.
Notice that 
\[\lambda_{u^{-1}}(q^{\frac{1}{2}})=1-\frac{\hbar}{4u}.\]
Hence, for any $\mathbf{v}\in\Lambda_{\mathfrak{c},d}$ and $f\in H^{G\times\mathbb{C}^*}_*(\mathscr{F}_\mathbf{v})$, \[\lambda_{u^{-1}}(E_{i,\mathbf{v}})\cdot f=\mathbf{H}_{i}(u)f.\] 
This finishes the proof.
\end{proof}

\subsection{The geometric $B$-operators}
Recall the generalized Steinberg variety
$$Z:=T^*\mathscr{F}\times_{\mathscr{N}}T^*\mathscr{F}=\bigsqcup\limits_{\mathbf{v},\mathbf{w}\in\Lambda_{\mathfrak{c},d}}T^*\mathscr{F}_{\mathbf{v}}\times_{\mathscr{N}}T^*\mathscr{F}_{\mathbf{w}}.$$ 
Let $Z_{\mathbf{v},\mathbf{w}}:=T^*\mathscr{F}_{\mathbf{v}}\times_{\mathscr{N}}T^*\mathscr{F}_{\mathbf{w}} $. The diagonal $G$-orbits on $\mathscr{F}_{\mathbf{v}}\times \mathscr{F}_{\mathbf{w}}$ are indexed by $\Xi_d(\mathbf{v},\mathbf{w})$, and 
\[Z_{\mathbf{v},\mathbf{w}}=\bigsqcup_{A\in \Xi_d(\mathbf{v},\mathbf{w})}Z_A,\]
where $Z_A:=T^*_{\mathcal{O}_A}(\mathscr{F}_{\mathbf{v}}\times \mathscr{F}_{\mathbf{w}})$ is the conormal bundle of the $G$-orbit corresponding to $A\in \Xi_d(\mathbf{v},\mathbf{w})$.

For any $\mathbf{v}\in \Xi_{\mathfrak{c},d}$ and any $1\leq i\leq 2n$, let 
\begin{equation}\label{equ:v'v''}
    \mathbf{v}'':=\mathbf{v}-\mathbf{e}_i-\mathbf{e}_{\tau i+1},\textit{ and }\mathbf{v}':=\mathbf{v}-\mathbf{e}_i+\mathbf{e}_{i+1}+\mathbf{e}_{\tau i}-\mathbf{e}_{\tau i+1}.
\end{equation}
The matrix $E^\theta_{i,i+1}(\mathbf{v}'',1)$ in \eqref{eq:Eijv} is minimal in the order \eqref{equ:order}. Hence, the $G$-orbit $\mathcal{O}_{E^\theta_{i,i+1}(\mathbf{v}'',1)}$ in $\mathscr{F} \times \mathscr{F}$ is closed, and it is given by 
\begin{align*}
\mathcal{O}_{E_{i,i+1}^{\theta}(\mathbf{v}'',1)}&=\bigg\{(F,F')\mid 
\substack{F=(V_k)_{0\le k \le N} 
\in \mathscr{F}_{\mathbf{v} }\\ F'=(V'_k)_{0\le k \le N}
\in \mathscr{F}_{\mathbf{v}' }}, V'_i\stackrel{1}{\subset} V_i, V_k=V'_k \textit{ if } k\neq i,\tau i\bigg\}.
\end{align*}
We use $\mathcal{L}_{\mathbf{v},i}$ to denote the tautological line bundle on $\mathcal{O}_{E_{i,i+1}^{\theta}(\mathbf{v}'',1)}$ whose stalk at a point $(F,F')$ above is given by $V_i/V'_i$. Under the isomorphism $H_*^{G\times \C^*}(\mathcal{O}_{E_{i,i+1}^{\theta}(\mathbf{v}'',1)})\simeq \mathbf{R}^{[E_{i,i+1}^{\theta}(\mathbf{v}'',1)]^\mathfrak{c}}$ in Proposition \ref{prop:pushforward}, $c_1(\mathcal{L}_{\mathbf{v},i})=x_{\bar{v}_i}$. 

We denote by $p_1$ and $p_2$ the two projections from the orbit $\mathcal{O}_{E_{i,i+1}^{\theta}(\mathbf{v}'',1)}$ to $\mathscr{F}_v$ and $\mathscr{F}_{v'}$, respectively. For $i\neq n+1$, the fiber $p_1^{-1}(F)$ is the Grassmannian $\Gr(v_i,V_{i}/V_{i-1})$. Hence,
\begin{equation}\label{equ:Eu3}
    \Eu(T^*_{p_1})=\prod_{\bar{v}_{i-1}<t< \bar{v}_i}(x_t-x_{\bar{v}_i})\in H_*^{G\times \C^*}(\mathcal{O}_{E_{i,i+1}^{\theta}(\mathbf{v}'',1)})\simeq \mathbf{R}^{[E_{i,i+1}^{\theta}(\mathbf{v}'',1)]^\mathfrak{c}}.
\end{equation}
If $i=n+1$, the fiber $p_1^{-1}(F)$ is 
\begin{align*}
    \bigg\{ F'=& (V'_k)_{0\le k \le N}\in \mathscr{F}_{\mathbf{v}'}\mid V_n\stackrel{1}{\subset} V'_n\subset V'_{n+1}=(V'_n)^\perp\subset V_{n+1}=V_n^\perp,\\
    &V_k=V'_k \textit{ if } k\neq n, n+1\bigg\}.
\end{align*}
Therefore,
\begin{align}\label{equ:Eu2}
    \Eu(T^*_{p_1})
    =&\prod_{\bar{v}_n+1 < t \leq \bar{v}_{n+1}}(x_{\bar{v}_{n}+1}-x_t)\in H_*^{G\times \C^*}(\mathcal{O}_{E_{n+1,n+2}^{\theta}(\mathbf{v}'',1)})
    \simeq \mathbf{R}^{[E_{n+1,n+2}^{\theta}(\mathbf{v}'',1)]^\mathfrak{c}}.
\end{align}

Recall that $Z_{E_{i,i+1}^{\theta}(\mathbf{v}'',1)}\subset Z$ is the conormal bundle of the $G$-orbit $\mathcal{O}_{E_{i,i+1}^{\theta}(\mathbf{v}'',1)}$ inside $\mathscr{F}\times \mathscr{F}$.
Define $\mathscr{B}_{i,\mathbf{v},r}\in H_{*}^{G\times \mathbb{C}^*}(Z)$ by the following formula
$$\mathscr{B}_{i,\mathbf{v},r}:=\sqrt{-1}(-1)^{v_{i}-1+r} \pi^*\bigg(c_1(\mathcal{L}_{\mathbf{v}, i})+\frac{n-i}{2}\hbar+\frac{\hbar}{4}\bigg)^r\cdot[Z_{E_{i,i+1}^{\theta}(\mathbf{v}'',1)}],$$
and set 
\begin{equation}\label{equ:geoBir}
    \mathscr{B}_{i,r}=\sum\limits_{\mathbf{v}''\in\Lambda_{\mathfrak{c},d-1}}\mathscr{B}_{i,\mathbf{v},r}\in H_{*}^{G\times \mathbb{C}^*}(Z).
\end{equation}

This will geometrize the operator $B_{i,r}$ in \eqref{equ:Bope}. Recall $H_{*}^{G\times \mathbb{C}^*}(Z)$ acts faithfully on $H_*^{G\times \C^*}(\mathscr{F})$, see Lemma \ref{lemma:faithful}.

\begin{prop}\label{prop:Bir}
    Under the isomorphism in Proposition \ref{prop:pushforward}, the convolution action of $\mathscr{B}_{i,r}\in H_{*}^{G\times \mathbb{C}^*}(Z)$ on $H_*^{G\times \C^*}(\mathscr{F})\simeq \mathbf{P}$ is given by the above operator $B_{i,r}$ in \eqref{equ:Bope}.
\end{prop}
\begin{proof}
     For any $f\in \mathbf{R}^{[\mathbf{v}']}[\hbar]$, $\mathscr{B}_{i,r}\star f= \mathscr{B}_{i,\mathbf{v},r}\star f$. From the proof of Proposition \ref{prop:pushforward}, we see that $p_2^*$ is the natural inclusion map under the isomophism in Proposition \ref{prop:pushforward}(1). Hence, $p_2^*(f)=f$. Let us first assume $i\neq n+1$, then
    \begin{align*}
        &\mathscr{B}_{i,\mathbf{v},r}\star f\\
        =&\sqrt{-1}(-1)^{v_{i}-1}p_{1*}\bigg(\Eu(q^{-2}T^*_{p_1})\cdot (-c_1(\mathcal{L}_{\mathbf{v},i})-\frac{n-i}{2}\hbar-\frac{\hbar}{4})^r\cdot p_2^*f\bigg)\\
        =&\sqrt{-1}(-1)^{v_{i}-1}W_{[\mathbf{v}^\mathfrak{c}]}/W_{[E_{i,i+1}^{\theta}(\mathbf{v}'',1)]^{\mathfrak{c}}}\bigg(\frac{\Eu(q^{-2}T^*_{p_1})}{\Eu(T_{p_1})}\cdot (-x_{\bar{v}_i}-\frac{n-i}{2}\hbar-\frac{\hbar}{4})^r\cdot f\bigg)\\
        =&\sqrt{-1}S_{[\bar{v}_{i-1}+1,\bar{v}_{i}]}/S_{[\bar{v}_{i-1}+1,\bar{v}_{i}-1]}\bigg(\prod_{\bar{v}_{i-1}<t\leq \bar{v}_{i}-1}\frac{\hbar+x_t-x_{\bar{v}_i}}{x_t-x_{\bar{v}_i}}\cdot (-x_{\bar{v}_i}-\frac{n-i}{2}\hbar-\frac{\hbar}{4})^r\cdot f\bigg)\\
        =&\sqrt{-1}\sum_{j\in [\mathbf{v}]_{i}}(j,\bar{v}_i)\bigg(\prod_{\bar{v}_{i-1}<t\leq \bar{v}_{i}-1}\frac{\hbar+x_t-x_{\bar{v}_i}}{x_t-x_{\bar{v}_i}}\cdot (-x_{\bar{v}_i}-\frac{n-i}{2}\hbar-\frac{\hbar}{4})^r\cdot f\bigg)\\
        =&\sqrt{-1}\sum_{j\in [\mathbf{v}]_{i}}\prod_{\bar{v}_{i-1}<t\leq \bar{v}_{i},t\neq j}\frac{\hbar+x_t-x_{j}}{x_t-x_{j}}\cdot (-x_{j}-\frac{n-i}{2}\hbar-\frac{\hbar}{4})^r\cdot (j,\bar{v}_i)f\\
        =&\sqrt{-1}\sum_{j\in [\mathbf{v}]_{i}}(-x_j-\frac{n-i}{2}\hbar-\frac{\hbar}{4})^r\cdot \Phi_{[\mathbf{v}]_{i}\setminus\{j\}}(x_j)\cdot (j,\bar{v}_i) f.
    \end{align*}
    Since $f\in \mathbf{R}^{[\mathbf{v}']}[\hbar]$, for $j\in [\mathbf{v}]_{i}$,
    \[((j,\bar{v}_i) f)(x_{[\mathbf{v}]})=f(x_{\tau^+_j[\mathbf{v}]}).\]
    This finishes the proof for the case $i\neq n+1$.

    Let us consider the case $i=n+1$. By definition, $E_{n+1,n+2}^{\theta}(\mathbf{v}'',1)=E_{n+1,n}^{\theta}(\mathbf{v}'',1)$. Hence, by the same argument as above, 
    \begin{align*}
        &(\mathscr{B}_{n+1,r}\star f)(x_{[\mathbf{v}]})\\
        =&(\mathscr{B}_{n+1,\mathbf{v},r}\star f)(x_{[\mathbf{v}]})\\
        =&\sqrt{-1}(-1)^{v_{n+1}-1}W_{[\mathbf{v}^\mathfrak{c}]}/W_{[E_{n+1,n}^{\theta}(\mathbf{v}'',1)]^{\mathfrak{c}}}\bigg(\frac{\Eu(q^{-2}T^*_{p_1})}{\Eu(T_{p_1})}\cdot (-x_{\bar{v}_{n+1}}+\frac{\hbar}{4})^r\cdot f\bigg)(x_{[\mathbf{v}]})\\
        =&\sqrt{-1}\sum_{j\in [\mathbf{v}]_{n+1}^\mathfrak{c}}(j,1+\bar{v}_n)\bigg(\prod_{\bar{v}_n+1<t\leq \bar{v}_{n+1}}\frac{\hbar+x_{\bar{v}_n+1}-x_t}{x_{\bar{v}_n+1}-x_t}\cdot (x_{1+\bar{v}_n}+\frac{\hbar}{4})^r\cdot f\bigg)(x_{[\mathbf{v}]})\\
        &+\sqrt{-1}\sum_{j\in [\mathbf{v}]_{n+1}^\mathfrak{c}}\iota_j(j,1+\bar{v}_n)\bigg(\prod_{\bar{v}_n+1<t\leq \bar{v}_{n+1}}\frac{\hbar+x_{\bar{v}_n+1}-x_t}{x_{\bar{v}_n+1}-x_t}\cdot (x_{1+\bar{v}_n}+\frac{\hbar}{4})^r\cdot f\bigg)(x_{[\mathbf{v}]})\\
        =&\sqrt{-1}\sum_{j\in [\mathbf{v}]_{n+1}} \prod_{\bar{v}_n+1\leq t\leq \bar{v}_{n+1},t\neq j}\frac{\hbar+x_{j}-x_t}{x_{j}-x_t}\cdot (x_{j}+\frac{\hbar}{4})^r\cdot f(x_{\tau^-_j[\mathbf{v}]})\\
        =&\sqrt{-1}\sum_{j\in [\mathbf{v}]_{n+1}} \prod_{\bar{v}_n+1\leq t\leq \bar{v}_{n+1},t\neq j}\frac{\hbar+x_{j'}-x_{t'}}{x_{j'}-x_{t'}}\cdot (x_{j'}+\frac{\hbar}{4})^r\cdot f(x_{\tau^-_{j'}[\mathbf{v}]})\\
        =&\sqrt{-1}\sum_{j\in [\mathbf{v}]_{n+1}} (-x_{j}+\frac{\hbar}{4})^r\cdot \Phi_{[\mathbf{v}_{n+1}]\setminus\{j\}}(x_j)\cdot f(x_{\tau^+_j[\mathbf{v}]}).
    \end{align*}
    Here the third equality follows from \eqref{equ:Eu2}, $x_{1+\bar{v}_n}=-x_{\bar{v}_{n+1}}$, and the fact that a set of representatives for $W_{[\mathbf{v}^\mathfrak{c}]}/W_{[E_{n+1,n}^{\theta}(\mathbf{v}'',1)]^{\mathfrak{c}}}$ is 
    \[\{(1+\bar{v}_n,j),\iota_{j}(1+\bar{v}_n,j))\mid 1+\bar{v}_n\leq j\leq d\}.\]
    The forth one follows from the change of variable $j\mapsto 2d+1-j=j'$ for the second summand, and  
    \[((\bar{v}_n+1,j)f)(x_{[\mathbf{v}]})=f(x_{\tau^-_j[\mathbf{v}]}),\quad \textit{ and } (\iota_j(\bar{v}_n+1,j)f)(x_{[\mathbf{v}]})=f(x_{\tau_{2d+1-j}^-[\mathbf{v}]}),\]
    for any $j\in [\mathbf{v}]_{n+1}^\mathfrak{c}$, $f\in \mathbf{R}^{[\mathbf{v}']^{\mathfrak{c}}}[\hbar]$. Finally, the fifth one follows from the change of variables $j\mapsto j'$ and $t\mapsto t'$.
    This finishes the proof.
\end{proof}

\subsection{Involution}
There is an involution on $Z$ obtained by switching the two factors $\mathscr{M}\times \mathscr{M}$, which induces the map $\varphi:\gamma\mapsto \gamma^t$ on $H_*^{G\times\mathbb{C}^*}(Z)$. For the convolution product, we have (see \cite{CG97})
\begin{equation*}\label{equ:inv}
    (\gamma_1\star\gamma_2)^t = \gamma_2^t\star \gamma_1^t.
\end{equation*}
Hence, this defines an algebra anti-involution $\varphi$ on $H_*^{G\times\mathbb{C}^*}(Z)$. Recall the $\mathbf{v}',\mathbf{v}''$ defined in \eqref{equ:v'v''}.
\begin{lem}\label{lem:inv}
    The following holds
    \[\varphi(\mathscr{H}_{i,r})=\mathscr{H}_{i,r},\textit{ and }\varphi(\mathscr{B}_{i,\mathbf{v},r})=(-1)^{v_i+v_{\tau i}+r}\mathscr{B}_{\tau i,\mathbf{v}',r}.\]
\end{lem}
\begin{proof}
    The first one holds as the classes $\mathscr{H}_{i,r}$ live on the diagonal copy $T^*\mathscr{F}$. Let us prove the second equation. First of all, $E_{i,i+1}^{\theta}(\mathbf{v}'',1)=E_{\tau i+1,\tau i}^{\theta}(\mathbf{v}'',1)$. Notice that under the switching of factors, the $G$-orbit $\mathcal{O}_{E_{i,i+1}^{\theta}(\mathbf{v}'',1)}\subset \mathscr{F}_\mathbf{v}\times \mathscr{F}_{\mathbf{v}'}$ becomes $\mathcal{O}_{E_{\tau i, \tau i+1}^{\theta}(\mathbf{v}'',1)}\subset \mathscr{F}_{\mathbf{v}'}\times \mathscr{F}_{\mathbf{v}}$, and the tautological line bundle $\mathcal{L}_{\mathbf{v},i}$ on $\mathcal{O}_{E_{i,i+1}^{\theta}(\mathbf{v}",1)}$ becomes the line bundle on $\mathcal{O}_{E_{\tau i, \tau i+1}^{\theta}(\mathbf{v}'',1)}$ whose fiber over a point $(F'=(V_k'),F=(V_k))$ is $V_i/V_i'$. Recall in the definition of $\mathscr{B}_{\tau i,\mathbf{v}',r}$, we used the line bundle $\mathcal{L}_{\mathbf{v}',\tau i}$, whose fiber over a point $(F'=(V_k'),F=(V_k))$ is $V_{\tau i}'/V_{\tau i}$. Since $V_i^\perp=V_{\tau i}$ and $(V'_i)^\perp=V'_{\tau i}$. Therefore, 
    \[\varphi(c_1(\mathcal{L}_{\mathbf{v},i}))=-c_1(\mathcal{L}_{\mathbf{v}',\tau i}).\]
    Hence, the lemma follows.
\end{proof}

\section{Main results}\label{sec:main}

In this section, we formulate the main result of this paper, which give a geometric realization of the twisted Yangian $\mathbf{Y}^\imath$ of quasi-split type $\AIII_{2n}^{(\tau)}$ via the equivariant Borel--Moore homology of the Steinberg variety in Section \ref{steinberg variety of type c}.

\subsection{Twisted Yangian}

Let $(c_{ij})_{2n\times 2n}$ be the Cartan matrix of type $A_{2n}$. As before, for $1\leq i\leq 2n$, $\tau i:=N-i=2n+1-i$. The following is the Drinfeld new presentation for the quasi-split twisted Yangian of type $\AIII_{2n}^{(\tau)}$, see \cite{LZ25}. We write $\{x,y\}=xy+yx$.
\begin{definition}\label{def:Y}
    The twisted Yangian of quasi-split type $\AIII_{2n}^{(\tau)}$, denoted by $\mathbf{Y}^\imath$, is the $\mathbb{C}[\hbar]$-algebra generated by $h_{i,r},b_{i,r},1\leq i\leq 2n,r\in\mathbb{N},$ subject to the following relations
    \begin{align*}
        &[h_{i,r},h_{j,s}]=0,\qquad h_{\tau i,0}=-h_{i,0},\\&[h_{i,0},b_{j,r}]=(c_{ij}-c_{\tau i,j})b_{j,r},\\&[h_{i,1},b_{j,r}]=(c_{ij}+c_{\tau i,j})b_{j,r+1}+\frac{\hbar(c_{ij}-c_{\tau i,j})}{2}\{h_{i,0},b_{j,r}\},\\&[h_{i,r+2},b_{j,s}]-[h_{i,r},b_{j,s+2}]\\\quad&=\frac{c_{ij}-c_{\tau i,j}}{2}\hbar\{h_{i,r+1},b_{j,s}\}+\frac{c_{ij}+c_{\tau i,j}}{2}\hbar\{h_{i,r},b_{j,s+1}\}+\frac{c_{ij}c_{\tau i,j}}{4}\hbar^2[h_{i,r},b_{j,s}],\\&
        [b_{i,r+1},b_{j,s}]-[b_{i,r},b_{j,s+1}]=\frac{c_{ij}}{2}\hbar\{b_{i,r},b_{j,s}\}-2\delta_{\tau i,j}(-1)^rh_{j,r+s+1},
    \end{align*}
    and the Serre relations: for $c_{ij}=0$,
    $$[b_{i,r},b_{j,s}]=\delta_{\tau i,j}(-1)^rh_{j,r+s},$$
    and for $c_{ij}=-1,j\neq\tau i\neq i$,
    $$\mathrm{Sym}_{k_1,k_2}[b_{i,k_1},[b_{i,k_2},b_{j,r}]]=0,$$
    and for $c_{i,\tau i}=-1$,$$\mathrm{Sym}_{k_1,k_2}[b_{i,k},[b_{i,k_2},b_{\tau i,r}]]=\frac{4}{3}\mathrm{Sym}_{k_1,k_2}(-1)^{k_1}\sum\limits_{p=0}^{k_1+r}3^{-p}[b_{i,k_2+p},h_{\tau i,k_1+r-p}],$$ where $h_{i,s}=0$ if $s<0$.
\end{definition}
It follows from the definition that for any $r\geq 0$, $h_{\tau i,r}=(-1)^{r+1}h_{i,r}$.
Now we restate these relations in terms of generating functions. Let 
$$b_{i}(u):=\hbar\sum\limits_{r\ge0}b_{i,r}u^{-r-1},\quad h_{i}(u):=1+\hbar\sum\limits_{r\ge0}h_{i,r}u^{-r-1},\quad h^{\circ}_{i}(u):=h_i(u)-1.$$
\begin{prop}\cite[Theorem 3.9]{LZ25}\label{prop:Ygen}
    The defining relations for the twisted Yangian $\mathbf{Y}^\imath$ of quasi-split type $\AIII_{2n}^{(\tau)}$ can be reformulated as follows:
    \begin{align*}
    [h_i(u),h_j(v)]=0,\qquad h_{\tau i}(u)=h_i(-u),
    \end{align*}
    \begin{align}\label{equ:bibj}
    (u-v)[b_i(u),b_j(v)]=&\frac{c_{ij}}{2}\hbar\{b_i(u),b_j(v)\}+\hbar([b_{i,0},b_j(v)]-[b_i(u),b_{j,0}])\\&-\delta_{\tau i,j}\hbar\bigg(\frac{2u}{u+v}h^{\circ}_i(u)+\frac{2v}{u+v}h^{\circ}_j(v)\bigg),\notag
    \end{align}
    \begin{align}\label{equ:hibj}
    (u^2-v^2)[h_i(u),b_j(v)]=&\frac{c_{ij}-c_{\tau i,j}}{2}\hbar u\{h_i(u),b_j(v)\}+\frac{c_{ij}+c_{\tau i,j}}{2}\hbar v\{h_i(u),b_j(v)\}\\&+\frac{c_{ij}c_{\tau i,j}}{2}\hbar^2[h_i(u),b_j(v)]-\hbar[h_{i}(u),b_{j,1}]\notag\\
    &-\hbar v[h_i(u),b_{j,0}]-\frac{c_{ij}+c_{\tau i,j}}{2}\hbar^2\{h_i(u),b_{j,0}\},\notag
    \end{align}
    \begin{align}\label{equ:bibjcij0}
    (u+v)[b_i(u),b_j(v)]=\delta_{\tau i,j}\hbar(h_j(v)-h_i(u)),\qquad (c_{ij}=0),
    \end{align}
    \begin{align}\label{equ:fserreij}
    [b_{i,0},[b_{i,0},b_{j,0}]]=0,\qquad (c_{ij}=-1,j\neq \tau i),
    \end{align}
    \begin{align}\label{equ:fserre}
    [b_{i,0},[b_{i,0},b_{\tau i,0}]]=4b_{i,0},\qquad (i=n,\textit{ or } n+1).
    \end{align}
\end{prop}
\begin{rem}
    It is proved in \cite[Proposition 3.12]{LZ25} that the relation \eqref{equ:fserre} and the other relations will imply the last relation in Definition \ref{def:Y}. 
\end{rem}

\subsection{The algebra homomorphism}
Recall $\mathbf{P}:=\bigoplus_{\mathbf{v}\in \Lambda_{\mathfrak{c},d}}\mathbf{R}^{[\mathbf{v}]^\mathfrak{c}}[\hbar]\simeq H_*^{G\times \C^*}(\mathscr{F})$. In Section \ref{sec:operators}, we have defined the operators $B_{i,r}, H_{i,r}$ on $\mathbf{P}$, and we also defined elements $\mathscr{B}_{i,r},\mathscr{H}_{i,r}$ in $H_*^{G\times \C^*}(Z)$, see \eqref{equ:Hivu},\eqref{equ:Bope},\eqref{equ:geoHir} and \eqref{equ:geoBir}.

\begin{thm}\label{thm:polyrep}
    The following assignment:
    \begin{align*}
    b_{i,r}\mapsto B_{i,r},\quad h_{i,r}\mapsto H_{i,r},\quad \textit{ for } 1\le i\le 2n, r\geq 0
    \end{align*}
    defines a representation of the twisted Yangian $\mathbf{Y}^\imath$ of quasi-split type $\AIII_{2n}^{(\tau)}$ on the space $\mathbf{P}$.
\end{thm}
We will prove this theorem in Section \ref{sec:relations} below by checking that the corresponding operators satisfy the relations in Proposition \ref{prop:Ygen}.
This representation is called the polynomial representation of $\mathbf{Y}^\imath$. With this, we can prove the main result of this paper. 
\begin{thm}\label{thm:main}
    There exists a unique algebra homomorphism $$\Psi:\mathbf{Y}^\imath\rightarrow H_*^{G\times \C^*}(Z)$$ by sending 
    \[b_{i,r}\mapsto \mathscr{B}_{i,r}, \quad h_{i,r}\mapsto \mathscr{H}_{i,r},\textit{ for } 1\le i\le 2n, r\geq 0.\]
\end{thm}
\begin{proof}
    It follows directly from Theorem \ref{thm:polyrep}, Lemma \ref{lemma:faithful}, and the explicit formulae in Propositions \ref{prop:Hir} and \ref{prop:Bir}.
\end{proof}

\section{Representations of twisted Yangian}\label{sec:rep}

In this section, we prove the surjectivity of the homomorphism $\Psi$ in Theorem \ref{thm:main} after certain specializations. Then via the general construction of representations of convolution algebras in \cite{CG97}, we get some representations of the twisted Yangian $\mathbf{Y}^\imath$ from the geometry of the cotangent bundle $\mathscr{M}$ of the partial flag varieties.

\subsection{Surjectivity of the homomorphism $\Psi_a$.}

The equivariant Borel--Moore homology group $H^{G\times\mathbb{C}^*}_*(Z)$ is a module over the base ring $H^{G\times\mathbb{C}^*}_*(\pt)\simeq\mathbf{R}^W[\hbar]$, whose characters correspond to semisimple conjugacy classes in $\mathfrak{g}\times\mathbb{C}$. 
Let $a:=(s,t)\in\mathfrak{g}\times\mathbb{C}$ be a semisimple element, let $\mathbb{C}_a$ be the one-dimensional module of $H_*^{G\times \mathbb{C}^*}(\mathrm{pt})\simeq \mathbb{C}[\mathfrak{g}]^G[\hbar]$ by evaluation at $a$, and let
$$H^{G\times\mathbb{C}^*}_*(Z)_a:=H^{G\times\mathbb{C}^*}_*(Z)\otimes_{H^{G\times\mathbb{C}^*}_*(\mathrm{pt})}\mathbb{C}_a.$$ 
The specialization of $\Psi$ in Theorem \ref{thm:main} at $\hbar=t$ gives an algebra homomorphism
\begin{equation}\label{surj mor}
    \Psi_a: \mathbf{Y}^\imath_t\to H_*^{G\times\mathbb{C}^*}(Z)_a.
\end{equation}
\begin{thm}\label{surj}
    Suppose that $t\neq0$. Then the homomorphism $\Psi_a$ in \eqref{surj mor} is surjective.
\end{thm}

The remainder of this subsection is devoted to the proof of Theorem \ref{surj}. We will assume $t\neq 0$ from now on. To that end, we consider the specialization of $H^{G\times\mathbb{C}^*}_*(Z)$ (and its localization) at $\hbar=t\neq0$, denoted by $H^{G\times\mathbb{C}^*}_*(Z)_t$, and the specialized morphism $\Psi_t$.

For any $\mathbf{v}\in\Lambda_{\mathfrak{c},d}$, let
$$\mathbf{I}_\mathbf{v}:=\bigg(\prod\limits_{i=1}^{n-1}\prod\limits_{\substack{m=-d\\m\neq v_{i}-v_{i+1}}}^{d}\frac{h_{i,0}-m\hbar}{(v_{i}-v_{i+1}-m)\hbar}\bigg)\cdot\prod\limits_{\substack{m=-d\\m\neq v_{n}-v_{n+1}}}^{d}\frac{h_{n,0}-(m-\frac{\hbar}{4})\hbar}{(v_{n}-v_{n+1}-m)\hbar}\in\mathbf{Y}^\imath_t.$$
\begin{lem}\label{ideom}
    For any $\mathbf{u},\mathbf{v},\mathbf{w}\in\Lambda_{\mathfrak{c},d}$, $\Psi_t(\mathbf{I}_\mathbf{v})\star H^{G\times\mathbb{C}^*}_*(Z_{\mathbf{uw}})\neq0$ if and only if $\mathbf{v}=\mathbf{u}$.
\end{lem}
\begin{proof}
    Note that, for $1\le i\le n$, $\mathscr{H}_{i,0}$ acts on $H^{G\times\mathbb{C}^*}_*(Z_{\mathbf{uw}})$ as scalar multiplication by $(u_{i}-u_{i+1}-\frac{1}{4}\delta_{i,n})\hbar$. Thus, for any $\gamma\in H^{G\times\mathbb{C}^*}_*(Z_{\mathbf{uv}})_t$,
    \begin{equation}\label{eq:1}
\Psi_t(\mathbf{I}_\mathbf{v})\star\gamma=\bigg(\prod\limits_{i=1}^{n}\prod\limits_{\substack{m=-d\\m\neq v_{i}-v_{i+1}}}^{d}\frac{u_{i}-u_{i+1}-m}{v_{i}-v_{i+1}-m}\bigg)\cdot\gamma.
    \end{equation}
    Hence, if $\mathbf{v}=\mathbf{u}$, the RHS of \eqref{eq:1} equals $\gamma$. If $\Psi_t(\mathbf{I}_{\mathbf{v}})\star\gamma\neq0$, we must have $u_{i}-u_{i+1}=v_{i}-v_{i+1}$, for $1\le i\le n$. Since $\sum_{i=1}^{n}u_i+\frac{u_{n+1}}{2}=\sum_{i=1}^{n}v_i+\frac{v_{n+1}}{2}=d$, we get $\mathbf{u}=\mathbf{v}$.
\end{proof}
Therefore, the classes $\mathscr{B}_{i,\mathbf{v},r}$ and $\mathscr{H}_{i,\mathbf{v},r}$ defined in Section \ref{sec:operators} belong to the image of $\Psi_t$. 
\begin{lem}\label{gene H}
    Fix $\mathbf{v}\in\Lambda_{\mathfrak{c},d}$ and $t\neq0$. Then the elements $\mathscr{H}_{i,\mathbf{v},k}\cdot[Z_{\diag(\mathbf{v})}]$ and $\sum_{j=1}^dx_j^{2k}$, for $1\le i\le n$ and $k\ge 1$, generate the algebra $H^{G\times\mathbb{C}^*}_*(Z_{\diag(\mathbf{v})})_t\simeq\mathbf{R}^{[\mathbf{v}]^\mathfrak{c}}$. 
\end{lem}
\begin{proof}
    Define classes $\hat{\mathscr{H}}_{i,r}$ by the following identity: $$\mathscr{H}_i(z)=1+\hbar\sum\limits_{r\ge0}\mathscr{H}_{i,r}z^{-r-1}=\mathrm{exp}\bigg(\hbar\sum\limits_{r\ge0
    }\hat{\mathscr{H}}_{i,r}(-z-\frac{n-i}{2}\hbar-\frac{\hbar}{4})^{-r-1}\bigg).$$
    Therefore, $\hat{\mathscr{H}}_{i,r}$ is a polynomial in $\{\mathscr{H}_{i,k}\}$.
    Recall that, for $1\le i\le n$, the class $\mathscr{H}_{i}(z)$ acts on $\mathbf{R}^{[\mathbf{v}]^\mathfrak{c}}[\hbar]$ as the scalar multiplication by $$(1-\frac{\delta_{i,n}}{4z}\hbar)\cdot \Phi_{[\mathbf{v}]_{i}}(-z-\frac{n-i}{2}\hbar-\frac{\hbar}{4})\cdot \Phi_{[\mathbf{v}]_{\tau i}}(z+\frac{n-i}{2}\hbar+\frac{\hbar}{4}).$$ A direct computation shows that the class $\hat{\mathscr{H}}_{i,k}$ acts on $H^{G\times\mathbb{C}^*}_*(Z_{\diag(\mathbf{v})})_t\simeq\mathbf{R}^{[\mathbf{v}]^\mathfrak{c}}$ by the following scalar multiplications $$\hat{\mathscr{H}}_{i,\mathbf{v},k}=
\begin{cases}
   -\sum\limits_{j\in[\mathbf{v}]_i}x_j^k+\sum\limits_{j\in[\mathbf{v}]_{i+1}}x_j^k+\text{lower terms},  \quad &\text{for}\ 1\le i\le n-1;\\
    -\sum\limits_{j\in[\mathbf{v}]_n}x_j^k+\sum\limits_{j\in[\mathbf{v}]^{\mathfrak{c}}_{n+1}}(x_j^k+(-x_j)^k)+\text{lower terms}, &\text{for} \  i=n.
\end{cases}$$
Let $\mathbf{R}^{[\mathbf{v}]^\mathfrak{c}}_{\le k}:=\{f\in\mathbf{R}^{[\mathbf{v}]^\mathfrak{c}}\mid \deg(f)\le k\}$. We claim that $$\mathfrak{U}_k=\bigg\{\hat{\mathscr{H}}_{i,\mathbf{v},r}\mid 1\le r\le k,\ 1\le i\le n \bigg\}\cup\bigg\{\sum_{j=1}^{d}x^{2l}_{j}\mid l\in\mathbb{N},\ 1\le l\le[\frac{k}{2}]\bigg\}$$ generates the vector space $\mathbf{R}^{[\mathbf{v}]^\mathfrak{c}}_{\le k}$. 
We prove this claim by induction on $k$. When $k=0$, there is nothing to prove. Suppose it holds for $k-1$. Since $\mathbf{R}^{[\mathbf{v}]^\mathfrak{c}}_{\le k}$ is generated by $$\bigg\{\sum_{j\in[\mathbf{v}]_i}x_j^s\mid 1\le i\le n,\ 1\le s\le k\bigg\}\cup\bigg\{\sum_{j\in[\mathbf{v}]^\mathfrak{c}_{n+1}}x_j^{2l}\mid 1\le l\le[\frac{k}{2}]\bigg\},$$ it remains to show that $\mathfrak{U}_k$ generates $$\bigg\{\sum_{j\in[\mathbf{v}]_i}x_j^k\mid 1\le i\le n\bigg\}\cup\bigg\{\frac{1+(-1)^k}{2}\sum_{j\in[\mathbf{v}]^\mathfrak{c}_{n+1}}x_j^{k}\bigg\}.$$  
Let $$\hat{\mathscr{H}}'_{i,\mathbf{v},k}=
\begin{cases}
   -\sum\limits_{j\in[\mathbf{v}]_i}x_j^k+\sum\limits_{j\in[\mathbf{v}]_{i+1}}x_j^k,  \quad &\text{for}\ 1\le i\le n-1;\\
    -\sum\limits_{j\in[\mathbf{v}]_n}x_j^k+\sum\limits_{j\in[\mathbf{v}]^{\mathfrak{c}}_{n+1}}(x_j^k+(-x_j)^k), &\text{for} \  i=n,
\end{cases}$$ which is the leading term of $\hat{\mathscr{H}}_{i,\mathbf{v},k}$.
Since $\mathfrak{U}_{k-1}\subset\mathfrak{U}_k$, by induction, $\hat{\mathscr{H}}'_{i,\mathbf{v},k}$ can be generated by $\mathfrak{U}_k$.
If $k$ is odd, then for $1\leq i\leq n$, $\sum\limits_{j\in[\mathbf{v}]_i}x^k_j=-\sum\limits_{s=i}^{n}\hat{\mathscr{H}}'_{s,\mathbf{v},k}$. If $k$ is even, then the column vector $$(\hat{\mathscr{H}}'_{1,\mathbf{v},k},\cdots,\hat{\mathscr{H}}'_{n,\mathbf{v},k},\sum\limits_{j=1}^dx_j^k)^T$$ is related to the vector $$(\sum\limits_{j\in[\mathbf{v}]_i}x^k_j,\cdots,\sum\limits_{j\in[\mathbf{v}]_{n}}x_j^k,\sum\limits_{j\in[\mathbf{v}]^\mathfrak{c}_{n+1}}x^k_j)^T$$
by the matrix $$A=\begin{pmatrix}
    -1&1&&&&&\\
    &-1&1&&&&\\
    &&-1&1&&&\\
    &&&\ddots&\ddots&&\\
    &&&&-1&1&\\
    &&&&&-1&2\\
    1&1&1&\cdots&1&1&1
\end{pmatrix}$$
Since $\mathrm{det}(A)=(-1)^{n}(2n+1)\neq0$, this completes the proof of the claim and hence the lemma. 
\end{proof}

\begin{lem}\label{gene B}
    Let $\mathbf{v}\in\Lambda_{\mathfrak{c},d-1}$.
    For $1\le i\le 2n$, $H^{G\times\mathbb{C}^*}_*(Z_{E_{i,i+1}^{\theta}(\mathbf{v},1)})$ is contained in the algebra generated by $\mathscr{B}_{i,\mathbf{v},0}$, and the classes supported on the diagonal orbits. 
\end{lem}
\begin{proof}
It suffices to show the statement for $1\le i\le n$, as the other cases can be proved similarly. First of all, the orbit $\mathcal{O}_{E_{i+1,i}^{\theta}(\mathbf{v},1)}$ is closed, and by Proposition \ref{prop:pushforward}, we have $$H^{G\times\mathbb{C}^*}_*(Z_{E_{i,i+1}^{\theta}(\mathbf{v},1)})\simeq\mathbf{R}^{[E_{i,i+1}^{\theta}(\mathbf{v},1)]^\mathfrak{c}}[\hbar],$$
where $W_{[E_{i,i+1}^{\theta}(\mathbf{v},1)]^\mathfrak{c}}$$$=\begin{cases}
    S_{[1,\bar{v}_1]}\times\cdots\times S_{[\bar{v}_{i-1}+1,\bar{v}_i]}\times S_{[\bar{v}_i+2,\bar{v}_{i+1}]}\times\cdots\times(\mathbb{Z}_2^{\frac{v_{n+1}}{2}}\rtimes S_{[\bar{v}_n+1,d]})&\text{for}\ 1\le i\le n-1, \\
    S_{[1,\bar{v}_1]}\times\cdots\times S_{[\bar{v}_{n-1}+1,\bar{v}_n]}\times (\mathbb{Z}_2^{\frac{v_{n+1}}{2}-1}\rtimes S_{[\bar{v}_n+2,d]})&\text{for}\ i=n.
\end{cases}$$
By the assumption of $\mathbf{v}$, $\mathbf{v}+\mathbf{e}_i+\mathbf{e}_{N+1-i}\in\Lambda_{\mathfrak{c},d}$. Moreover, we have 
\begin{align*}
    &H^{G\times\mathbb{C}^*}_*(Z_{\diag(\mathbf{v}+\mathbf{e}_i+\mathbf{e}_{N+1-i})})\\
    \simeq&\begin{cases}
\mathbf{R}^{S_{[1,\bar{v}_1]}\times\cdots\times S_{[\bar{v}_{i-1}+1,\bar{v}_i+1]}\times S_{[\bar{v}_i+2,\bar{v}_{i+1}]}\times\cdots\times(\mathbb{Z}_2^{\frac{v_{n+1}}{2}}\rtimes S_{[\bar{v}_n+1,d]})}[\hbar] & \textit{ if } 1\leq i\leq n-1,\\
\mathbf{R}^{S_{[1,\bar{v}_1]}\times\cdots\times S_{[\bar{v}_{n-1}+1,\bar{v}_n+1]}\times (\mathbb{Z}_2^{\frac{v_{n+1}}{2}-1}\rtimes S_{[\bar{v}_n+2,d]})}[\hbar] & \textit{ if } i=n.
\end{cases}
\end{align*}
Convolving with the class $\mathscr{B}_{i,\mathbf{v},0}$, we see that $$H^{G\times\mathbb{C}^*}_*(Z_{\diag(\mathbf{v}+\mathbf{e}_i+\mathbf{e}_{N+1-i})})\subset H^{G\times\mathbb{C}^*}_*(Z_{E_{i,i+1}^{\theta}(\mathbf{v},1)}).$$ Similarly, $H^{G\times\mathbb{C}^*}_*(Z_{\diag(\mathbf{v}+\mathbf{e}_{i+1}+\mathbf{e}_{N-i})})\subset H^{G\times\mathbb{C}^*}_*(Z_{E_{i,i+1}^{\theta}(\mathbf{v},1)})$.
Since $$\sum\limits_{j\in[\bar{v}_{i-1}+1,\bar{v}_{i}+1]}x_{j}^k\in H^{G\times\mathbb{C}^*}_{*}(Z_{\diag(\mathbf{v}+\mathbf{e}_i+\mathbf{e}_{N+1-i})})$$ and $$\sum\limits_{j\in[\bar{v}_{i-1}+1,\bar{v}_{i}]}x_{j}^k\in H^{G\times\mathbb{C}^*}_{*}(Z_{\diag(\mathbf{v}+\mathbf{e}_{i+1}+\mathbf{e}_{N-i})}),$$
we have $x^{k}_{\bar{v}_{i}+1}\in H^{G\times\mathbb{C}^*}_*(Z_{E_{i+1,i}^{\theta}(\mathbf{v},1)})$.
Together with the above discussion, the lemma follows from the fact that $H^{G\times\mathbb{C}^*}_*(Z_{E_{i,i+1}^{\theta}(\mathbf{v},1)})$ is generated by $H^{G\times\mathbb{C}^*}_*(Z_{\diag(\mathbf{v}+\mathbf{e}_i+\mathbf{e}_{N+1-i})})$ and $x_{\bar{v}_{i}+1}^k$ ($k\ge0$).
\end{proof}

Now we can finish the proof of Theorem \ref{surj}.
\begin{proof}
   Upon specialization at $s\in\mathfrak{g}$, the elements $\sum_{j=1}^{d}x_{j}^{2k}$ in Lemma \ref{gene H} specialize to scalars. Now Theorem \ref{surj} follows from Theorem \ref{thm:generators}, Lemma \ref{ideom}, Lemma \ref{gene H}, and Lemma \ref{gene B}.
\end{proof}

\subsection{Construction of representations}
Via the general construction of the representations of the convolutions in \cite[Chapter 8]{CG97}, we can get some representations of $\mathbf{Y}^\imath$ via the homomorphism \eqref{surj mor} as follows.

Let $A$ be the closed subgroup of $G\times\mathbb{C}^*$ generated by $\exp(a)$. Then $A$ is an algebraic torus, and we have natural isomorphism $H^A_*(\pt)\simeq\mathbb{C}[\mathrm{Lie}(A)]$. Suppose $\mathcal{T}$ is a maximal torus of $G\times\mathbb{C}^*$ which contains $A$, then $\mathrm{Lie}(A)\subset\mathrm{Lie}(\mathcal{T})$. The natural maps $H^{G\times\mathbb{C}^*}_*(\pt)\simeq\mathbb{C}[\mathrm{Lie}(\mathcal{T})]^W\hookrightarrow\mathbb{C}[\mathrm{Lie}(\mathcal{T})]\stackrel{res}{\to}\mathbb{C}[\mathrm{Lie}(A)]\simeq H^{A}(\pt)$ make $H^A_*(\pt)$ a $H^{G\times\mathbb{C}^*}_*(\pt)$-algebra. Then $H^{G\times\mathbb{C}^*}_*(\pt)$-module $\mathbb{C}_a$ can be regarded as an $H^A_*(\pt)$-module via evaluation at $a$ . We have algebra isomorphisms $$H_*^{G\times \mathbb{C}^*}(Z)_a:=H_*^{G\times \mathbb{C}^*}(Z)\otimes_{H^{G\times\mathbb{C}^*}_*(\pt)}\mathbb{C}_a\simeq H^A_*(Z)\otimes_{H^A_*(\mathrm{pt})}\mathbb{C}_a\simeq H_{*}(Z^A).$$ Here $H_{*}(Z^A)$ denotes the (non-equivariant) Borel-Moore homology of $Z^A$ with complex coefficients, which also has convolution algebra structure, see \cite[Charpter 2]{CG97}. The first isomorphism is due to the homological version of \cite[Theorem 6.2.10]{CG97}, while the second one is the homological version of bi-variant localization map \cite[Theorem 5.11.10]{CG97}. Composing with the surjective algebra homomorphism $\Psi_a$ from Theorem \ref{surj}, we get a surjective algebra homomorphism \begin{equation}\label{surj2}
    \mathbf{Y}^\imath_t\twoheadrightarrow H_{*}(Z^A).
\end{equation}
Therefore, every representation of the convolution algebra $H_{*}(Z^A)$ pulls back to a representation of $\mathbf{Y}^\imath_t$. Since the homomorphism in \eqref{surj2} is surjective, the pullback of irreducible representations will remain irreducible.

Recall $a=(s,t)\in\mathfrak{g}\times\mathbb{C}$ be a semisimple element with $t\neq0$. Let $G(s)$ be the centralizer of $s$.
Let $\mathscr{M}^A:=\bigsqcup_{\mathbf{v}}(T^*\mathscr{F}_\mathbf{v})^A$ be the fixed loci. Then the map $\pi: \mathscr{M}^A\to\mathscr{N}^A$ is $G(s)$-equivariant. For any $x\in \mathscr{N}^A$, let $\mathscr{M}^A_x$ be the fiber, which are called Spaltenstein varieties. The equivariant version of the decomposition theorem gives $$\pi_*\underline{\mathbb{C}}_{\mathscr{M}^A}=\bigoplus\limits_{k\in\mathbb{Z},\phi=(\mathcal{O}_{\phi}\subset\mathscr{N}^A,\chi_\phi)}L_\phi(k)\otimes\mathrm{IC}_\phi[k].$$
Here $\mathcal{O}_\phi$ is a $G(s)$-orbit on $\mathscr{N}^A$, $\chi_\phi$ is a $G(s)$-equivariant local system on $\mathcal{O}_\phi$, $\mathrm{IC}_\phi$ is the corresponding intersection homology complex, and $L_\phi(k)$ is some vector space.
Let $L_\phi=\oplus_{k}L_\phi(k)$, which is a simple $H_*(Z^A)$-module if nonzero, see \cite[Theorem 8.6.12]{CG97}. For any $x\in\mathcal{O}_\phi$, let $H_*(\mathscr{M}^A_x)_\phi$ be the $\phi$-isotypical component of $H_*(\mathscr{M}^A_x)$, which is a module of $H_*(Z^A)$ via convolution. Via the pullback \eqref{surj mor}, $H_*(\mathscr{M}^A_x)_\phi$ is also a module of $\mathbf{Y}^\imath_t$, which is called the {\it standard module}. We can also view $L_\phi$ as a $\mathbf{Y}^\imath_t$-module this way. Then the constructions in \cite[Chapter 8]{CG97} gives the following result.
\begin{thm}
    Assume $t\neq0$.
    \begin{itemize}
        \item[(1)] The $\mathbf{Y}^\imath_t$-module $L_\phi$ is simple if it is nonzero.
        \item[(2)] For any $\phi=(\mathcal{O}_\phi,\chi_\phi)$ and $\psi=(\mathcal{O}_\psi,\chi_\psi)$ and $x\in\mathcal{O}_\phi$,
        $$[H_*(\mathscr{M}_x^A)_\phi:L_\psi]=\sum\limits_{k}\dim H^k(i^!_x\mathrm{IC}_\psi)_\phi.$$
    \end{itemize}
\end{thm}
As mentioned in the introduction, the twisted Yangian $\mathbf{Y}^\imath$ studied in this paper is closely related to the reflection algebra introduced by Molev and Ragoucy, see \cite[Theorem 6.17]{LZ25} and \cite{MR02}. It is an interesting question to identify the representations constructed above with those from \cite{MR02}.

\section{Checking relations}\label{sec:relations}

In the rest of this paper, we prove Theorem \ref{thm:polyrep} by checking that the operators satisfy the relations in Proposition \ref{prop:Ygen} for the twisted Yangian $\mathbf{Y}^\imath$. 

\subsection{Operators}
Let us first recall the operators from Section \ref{sec:operators}. 
For $1\leq i\leq 2n$ and $\mathbf{v}\in\Lambda_{\mathfrak{c},d}$, 
\begin{align*}
        H_{i,\mathbf{v}}(u):=&(1-\frac{\delta_{i,n}-\delta_{i,n+1}}{4u}\hbar)\cdot \Phi_{[\mathbf{v}]_{i}}(-u-\frac{n-i}{2}\hbar-\frac{\hbar}{4})\cdot \Phi_{[\mathbf{v}]_{\tau i}}(u+\frac{n-i}{2}\hbar+\frac{\hbar}{4}).
\end{align*}
Consider $H_{i,\mathbf{v}}(u)$ as a series in $u^{-1}$, then it has constant $1$, and set $H^\circ_{i,\mathbf{v}}(u):=H_{i,\mathbf{v}}(u)-1$. Recall $\mathbf{H}_i(u)$ is the operator on $\mathbf{P}$, which acts on $\mathbf{R}^{[\mathbf{v}]^\mathfrak{c}}$ by multiplying by $H_{i,\mathbf{v}}(u)$.

For $1\leq i\leq 2n$,
\begin{equation*}
    (B_{i,r}f)(x_{[\mathbf{v}]}):=\sqrt{-1}\sum_{j\in [\mathbf{v}]_{i}}(-x_j-\frac{n-i}{2}\hbar-\frac{\hbar}{4})^r\cdot \Phi_{[\mathbf{v}]_{i}\setminus\{j\}}(x_j)\cdot  f(x_{\tau^+_j[\mathbf{v}]}),
\end{equation*}
and let
\[
\mathbf{B}_i(u):=\hbar\sum\limits_{r\ge0}B_{i,r}u^{-r-1}.
\]
Then 
\begin{align*}
   (\mathbf{B}_{i}(u)f)(x_{\mathbf{[v]}})=\sqrt{-1}\sum\limits_{t\in[\mathbf{v}]_{i}}\frac{\hbar}{u+x_t+\frac{n-i}{2}\hbar+\frac{\hbar}{4}}\Phi_{[\mathbf{v}]_{i}\backslash\{t\}}(x_t)\cdot f(x_{\tau_t^+\mathbf{[v]}}).
\end{align*}
Then under the map in Theorem \ref{thm:polyrep}, $h_i(u)$ (resp. $b_i(u)$) is sent to $\mathbf{H}_{i}(u)$ (resp. $\mathbf{B}_{i}(u)$.)

The first relation in Proposition \ref{prop:Ygen}
follows directly from the definition. Before checking other relations, we need some preparations.
For a Laurent series $p(z)=\sum_{i\in \mathbb{Z}} a_iz^{i}$ in $z^{-1}$,
we denote the truncation 
\[\big(p(z)\big)^{\circ}
:=\sum_{i< 0}a_iz^i.\] 
If $\tilde{p}(z):=p(z)/z$, then
\begin{equation}\label{equ:trundivbyz}
    u\big(\tilde{p}(u)\big)^\circ+v\big(\tilde{p}(-v)\big)^\circ
    =u\sum_{i< 1}a_iu^{i-1}-(-v)\sum_{i< 1}a_i(-v)^{i-1}=\big(p(u)\big)^{\circ}-\big(p(-v)\big)^{\circ}.
\end{equation}

\begin{lem}\cite{SSX25}\label{lem:partialfrac}
    Assume $p(z)=\dfrac{q(z)}{\prod_{i=1}^n (z-z_i)}$, where $z_1,\ldots,z_n$ are distinct and $q(z)\in \mathbb{C}[z]$. Then
\[(p(z))^\circ=\sum_{i=1}^n\frac{1}{z-z_i}\frac{q(z_i)}{\prod_{j\neq i}(z_i-z_j)}=\sum_{i=1}^n\frac{1}{z-z_i}\Res(p(z),z_i).\]
\end{lem}

Let us apply this lemma to $H_{i,\mathbf{v}}(u)$ in \eqref{equ:Hivu}. Since it is a series in $u^{-1}$ with constant term $1$, $H^\circ_{i,\mathbf{v}}(u)$ is its truncation $(H_{i,\mathbf{v}}(u))^\circ$ defined as above.
\begin{lem}\label{lem:Hexp}
    For $i\neq n, n+1$, the following identities hold:
    \begin{align*}
    H^\circ_{i,\mathbf{v}}(u)=&\sum\limits_{t\in[\mathbf{v}]_{i}}\frac{\hbar}{u+x_t+\frac{n-i}{2}\hbar+\frac{\hbar}{4}}\Phi_{[\mathbf{v}]_{i}\backslash\{t\}}(x_t)\Phi_{[\mathbf{v}]_{\tau i}}(-x_t)\\
    &-\sum\limits_{s\in[\mathbf{v}]_{\tau i}}\frac{\hbar}{u-x_s+\frac{n-i}{2}\hbar+\frac{\hbar}{4}}\Phi_{[\mathbf{v}]_{i}}(-x_s)\Phi_{[\mathbf{v}]_{\tau i}\backslash\{s\}}(x_s),
\end{align*}
\begin{align*}
    (u\cdot H_{i,\mathbf{v}}(u))^\circ=&-\sum\limits_{t\in[\mathbf{v}]_{i}}\frac{\hbar(x_t+\frac{n-i}{2}\hbar+\frac{\hbar}{4})}{u+x_t+\frac{n-i}{2}\hbar+\frac{\hbar}{4}}\Phi_{[\mathbf{v}]_{i}\backslash\{t\}}(x_t)\Phi_{[\mathbf{v}]_{\tau i}}(-x_t)\\
    &-\sum\limits_{s\in[\mathbf{v}]_{\tau i}}\frac{\hbar(x_s-\frac{n-i}{2}\hbar-\frac{\hbar}{4})}{u-x_s+\frac{n-i}{2}\hbar+\frac{\hbar}{4}}\Phi_{[\mathbf{v}]_{i}}(-x_s)\Phi_{[\mathbf{v}]_{\tau i}\backslash\{s\}}(x_s),
\end{align*}
and for $i=n$,
\begin{align*}
    (u\cdot H_{n,\mathbf{v}}(u))^\circ=&-\sum\limits_{t\in[\mathbf{v}]_{n}}\frac{\hbar(x_t+\frac{\hbar}{2})}{u+x_t+\frac{\hbar}{4}}\Phi_{[\mathbf{v}]_{n}\backslash\{t\}}(x_t)\Phi_{[\mathbf{v}]_{n+1}}(-x_t)\\
    &-\sum\limits_{s\in[\mathbf{v}]_{n+1}}\frac{\hbar(x_s-\frac{\hbar}{2})}{u-x_s+\frac{\hbar}{4}}\Phi_{[\mathbf{v}]_{n}}(-x_s)\Phi_{[\mathbf{v}]_{n+1}\backslash\{s\}}(x_s).
\end{align*}
\end{lem}
\begin{proof}
    Recall that for $i\neq n,n+1$,
    \begin{align*}
        H_{i,\mathbf{v}}(u):=&\Phi_{[\mathbf{v}]_{i}}(-u-\frac{n-i}{2}\hbar-\frac{\hbar}{4})\Phi_{[\mathbf{v}]_{\tau i}}(u+\frac{n-i}{2}\hbar+\frac{\hbar}{4})\\
        =&\prod_{t\in [\mathbf{v}]_{i}}\bigg(1+\frac{\hbar}{x_t+u+\frac{n-i}{2}\hbar+\frac{\hbar}{4}}\bigg)\prod_{s\in [\mathbf{v}]_{\tau i}}\bigg(1+\frac{\hbar}{x_s-u-\frac{n-i}{2}\hbar-\frac{\hbar}{4}}\bigg).
    \end{align*}
    It has simple poles at $\{-x_t-\frac{n-i}{2}\hbar-\frac{\hbar}{4}\mid t\in [\mathbf{v}]_{i}\}$ and $\{x_s-\frac{n-i}{2}\hbar-\frac{\hbar}{4}\mid s\in [\mathbf{v}]_{\tau i}\}$. Then the first two identities follows from Lemma \ref{lem:partialfrac}. Same arguments apply to  \begin{align*}
        u\cdot H_{n,\mathbf{v}}(u)=&(u-\frac{\hbar}{4})\Phi_{[\mathbf{v}]_{n}}(-u-\frac{\hbar}{4})\Phi_{[\mathbf{v}]_{\tau n}}(u+\frac{\hbar}{4})\\
    \end{align*}
    yields the last identity.
\end{proof}

\subsection{Relation \eqref{equ:bibj}}
We need to check the relation 
\begin{align}\label{equ:BBij}
    (u-v)[b_i(u),b_j(v)]-&\frac{c_{ij}}{2}\hbar\{b_i(u),b_j(v)\}-\hbar([b_{i,0},b_j(v)]-[b_i(u),b_{j,0}])\\&=-\delta_{\tau i,j}\hbar\bigg(\frac{2u}{u+v}h^{\circ}_i(u)+\frac{2v}{u+v}h^{\circ}_j(v)\bigg).\notag
\end{align}

If $j\neq \tau i$ and $c_{ij}=0$, then it follows from the relation $[\mathbf{B}_i(u),\mathbf{B}_j(v)]=0$ checked in Section \ref{sec:bibjcij0} below. Hence, we are left with the cases $j= \tau i$ or $c_{ij}\neq 0$, which is further decomposed into the following cases:
\begin{itemize}
    \item $j\neq \tau i$, $c_{ij}= 2$; 
    \item $j\neq \tau i$, $c_{ij}= -1$;
    \item $j=\tau i$. 
\end{itemize}

\subsubsection{The case $j\neq \tau i$, $i=j$}

By definition,
\begin{align*}
    &(\mathbf{B}_{i}(u)\mathbf{B}_{i}(v)f)(x_{\mathbf{[v]}})\\
    =&-\sum\limits_{t\neq s\in[\mathbf{v}]_{i}}\frac{\hbar}{u+x_t+\frac{n-i}{2}\hbar+\frac{\hbar}{4}}\frac{\hbar}{v+x_s+\frac{n-i}{2}\hbar+\frac{\hbar}{4}}\Phi_{[\mathbf{v}]_{i}\backslash\{t\}}(x_t)\Phi_{[\mathbf{v}]_{i}\backslash\{t,s\}}(x_s)\cdot (\mathbf{B}_{i}(v)f)(x_{\tau_s^+\tau_t^+\mathbf{[v]}}),
\end{align*}
and 
\begin{align*}
    &(\mathbf{B}_{i}(v)\mathbf{B}_{i}(u)f)(x_{\mathbf{[v]}})\\
    =&-\sum\limits_{t\neq s\in[\mathbf{v}]_{i}}\frac{\hbar}{u+x_t+\frac{n-i}{2}\hbar+\frac{\hbar}{4}}\frac{\hbar}{v+x_s+\frac{n-i}{2}\hbar+\frac{\hbar}{4}}\Phi_{[\mathbf{v}]_{i}\backslash\{s\}}(x_s)\Phi_{[\mathbf{v}]_{i}\backslash\{t,s\}}(x_t)\cdot (\mathbf{B}_{i}(v)f)(x_{\tau_s^+\tau_t^+\mathbf{[v]}}).
\end{align*}
Then \eqref{equ:BBij} holds as the 
coefficient of 
\[\frac{\hbar}{u+x_t+\frac{n-i}{2}\hbar+\frac{\hbar}{4}}\frac{\hbar}{v+x_s+\frac{n-i}{2}\hbar+\frac{\hbar}{4}}\Phi_{[\mathbf{v}]_{i}\backslash\{t,s\}}(x_t)\Phi_{[\mathbf{v}]_{i}\backslash\{t,s\}}(x_s)\cdot (\mathbf{B}_{i}(v)f)(x_{\tau_s^+\tau_t^+\mathbf{[v]}})\] in the left hand side of \eqref{equ:BBij} is
\begin{align*}
    &-(u-v-\hbar)\varphi(x_s-x_t)+(u-v+\hbar)\varphi(x_t-x_s)\\
    &+(u+x_t+\frac{n-i}{2}\hbar+\frac{\hbar}{4})\varphi(x_s-x_t)-(v+x_s+\frac{n-i}{2}\hbar+\frac{\hbar}{4})\varphi(x_s-x_t)\\
    &-(u+x_t+\frac{n-i}{2}\hbar+\frac{\hbar}{4})\varphi(x_t-x_s)+(v+x_s+\frac{n-i}{2}\hbar+\frac{\hbar}{4})\varphi(x_t-x_s)=0.
\end{align*}

\subsubsection{The case $j\neq \tau i$, $c_{i,j}=-1$}

In this case, $j=i\pm 1$. By Lemma \ref{lem:inv}, we can assume $j=i+1$. Since $j\neq \tau i$, we get $i\neq n$. Then
\begin{align*}
    &(\mathbf{B}_{i}(u)\mathbf{B}_{i+1}(v)f)(x_{\mathbf{[v]}})\\
    =&-\sum\limits_{t\in[\mathbf{v}]_{i}}\sum\limits_{s\in[\mathbf{v}]_{i+1}\cup\{t\}}\frac{\hbar}{u+x_t+\frac{n-i}{2}\hbar+\frac{\hbar}{4}}\frac{\hbar}{v+x_s+\frac{n-i}{2}\hbar-\frac{\hbar}{4}}\Phi_{[\mathbf{v}]_{i}\backslash\{t\}}(x_t)\Phi_{[\mathbf{v}]_{i+1}\cup\{t\}\backslash\{s\}}(x_s)\cdot f(x_{\tau_s^+\tau_t^+\mathbf{[v]}}),
\end{align*}
and
\begin{align*}
    &(\mathbf{B}_{i+1}(v)\mathbf{B}_{i}(u)f)(x_{\mathbf{[v]}})\\
    =&-\sum\limits_{t\in[\mathbf{v}]_{i}}\sum\limits_{s\in[\mathbf{v}]_{i+1}}\frac{\hbar}{u+x_t+\frac{n-i}{2}\hbar+\frac{\hbar}{4}}\frac{\hbar}{v+x_s+\frac{n-i}{2}\hbar-\frac{\hbar}{4}}\Phi_{[\mathbf{v}]_{i}\backslash\{t\}}(x_t)\Phi_{[\mathbf{v}]_{i+1}\backslash\{s\}}(x_s)\cdot f(x_{\tau_s^+\tau_t^+\mathbf{[v]}}),
\end{align*}
Therefore, the coefficient of 
\[\frac{\hbar}{u+x_t+\frac{n-i}{2}\hbar+\frac{\hbar}{4}}\frac{\hbar}{v+x_s+\frac{n-i}{2}\hbar-\frac{\hbar}{4}}\Phi_{[\mathbf{v}]_{i}\setminus\{t\}}(x_t)\Phi_{[\mathbf{v}]_{i+1}}(x_s)\cdot f(x_{\tau_t^-\tau_t^-[\mathbf{v}]})\]
in the left hand side of \eqref{equ:BBij}
is 
\begin{align*}
    -(u-v+\frac{\hbar}{2})+(u+x_t+\frac{n-i}{2}\hbar+\frac{\hbar}{4})-(v+x_t+\frac{n-i}{2}\hbar-\frac{\hbar}{4})=0.
\end{align*}
On the other hand, for $t\in[\mathbf{v}]_{i}$ and $s\in[\mathbf{v}]_{i+1}$ the coefficient of 
\[\frac{\hbar}{u+x_t+\frac{n-i}{2}\hbar+\frac{\hbar}{4}}\frac{\hbar}{v+x_s+\frac{n-i}{2}\hbar-\frac{\hbar}{4}}\Phi_{[\mathbf{v}]_{i}\backslash\{t\}}(x_t)\Phi_{[\mathbf{v}]_{i+1}\backslash\{s\}}(x_s)\cdot f(x_{\tau_s^+\tau_t^+\mathbf{[v]}})\]
in the left hand side of \eqref{equ:BBij}
is 
\begin{align*}
    &-(u-v+\frac{\hbar}{2})\varphi(x_t-x_s)+(u-v-\frac{\hbar}{2})\\
    &+(u+x_t+\frac{n-i}{2}\hbar+\frac{\hbar}{4})\varphi(x_t-x_s)-(u+x_t+\frac{n-i}{2}\hbar+\frac{\hbar}{4})\\
    &-(v+x_s+\frac{n-i}{2}\hbar-\frac{\hbar}{4})\varphi(x_t-x_s)+(v+x_s+\frac{n-i}{2}\hbar-\frac{\hbar}{4})=0.
\end{align*}
Hence, \eqref{equ:BBij} holds in this case.

\subsubsection{The case $j= \tau i$}

Let us first consider the case $c_{i,\tau i}=0$.
Then 
\begin{align*}
    &(\mathbf{B}_{i}(u)\mathbf{B}_{\tau i}(v)f)(x_{\mathbf{[v]}})\\
    =&-\sum\limits_{t\in[\mathbf{v}]_{i}}\sum\limits_{s\in[\mathbf{v}]_{\tau i}\cup\{t'\}}\frac{\hbar}{u+x_t+\frac{n-i}{2}\hbar+\frac{\hbar}{4}}\frac{\hbar}{v+x_s-\frac{n-i}{2}\hbar-\frac{\hbar}{4}}\Phi_{[\mathbf{v}]_{i}\backslash\{t\}}(x_t)\Phi_{[\mathbf{v}]_{\tau i}\cup\{t'\}\backslash\{s\}}(x_s)\cdot f(x_{\tau_s^+\tau_t^+\mathbf{[v]}}),
\end{align*}
and 
\begin{align*}
    &(\mathbf{B}_{\tau i}(v)\mathbf{B}_{i}(u)f)(x_{\mathbf{[v]}})\\
    =&-\sum\limits_{s\in[\mathbf{v}]_{\tau i}}\sum\limits_{t\in[\mathbf{v}]_{ i}\cup\{s'\}}\frac{\hbar}{u+x_t+\frac{n-i}{2}\hbar+\frac{\hbar}{4}}\frac{\hbar}{v+x_s-\frac{n-i}{2}\hbar-\frac{\hbar}{4}}\Phi_{[\mathbf{v}]_{i}\cup\{s'\}\backslash\{t\}}(x_t)\Phi_{[\mathbf{v}]_{\tau i}\backslash\{s\}}(x_s)\cdot f(x_{\tau_s^+\tau_t^+\mathbf{[v]}}).
\end{align*}
It is obvious that for $t\in[\mathbf{v}]_{i}, s\in[\mathbf{v}]_{\tau i}$, the coefficient of 
$f(x_{\tau_s^+\tau_t^+\mathbf{[v]}})$
in the left hand side of \eqref{equ:BBij} is $0$.
Therefore, the coefficient of $f(x_{\mathbf{[v]}})$ in the left hand side of \eqref{equ:BBij} is 
\begin{align*}
    &\sum\limits_{t\in[\mathbf{v}]_{i}}\Phi_{[\mathbf{v}]_{i}\backslash\{t\}}(x_t)\Phi_{[\mathbf{v}]_{\tau i}}(-x_t)\\
    &\bigg(-(u-v)\frac{\hbar}{u+x_t+\frac{n-i}{2}\hbar+\frac{\hbar}{4}}\frac{\hbar}{v-x_t-\frac{n-i}{2}\hbar-\frac{\hbar}{4}}+\hbar\frac{\hbar}{v-x_t-\frac{n-i}{2}\hbar-\frac{\hbar}{4}}-\hbar\frac{\hbar}{u+x_t+\frac{n-i}{2}\hbar+\frac{\hbar}{4}}\bigg)\\
    &+\sum\limits_{s\in[\mathbf{v}]_{\tau i}}\Phi_{[\mathbf{v}]_{i}}(-x_s)\Phi_{[\mathbf{v}]_{\tau i}\backslash\{s\}}(x_s)\\
    &\bigg((u-v)\frac{\hbar}{u-x_s+\frac{n-i}{2}\hbar+\frac{\hbar}{4}}\frac{\hbar}{v+x_s-\frac{n-i}{2}\hbar-\frac{\hbar}{4}}-\hbar\frac{\hbar}{v+x_s-\frac{n-i}{2}\hbar-\frac{\hbar}{4}}+\hbar\frac{\hbar}{u-x_s+\frac{n-i}{2}\hbar+\frac{\hbar}{4}}\bigg)\\
    =&\frac{2\hbar}{u+v}\bigg(\sum\limits_{t\in[\mathbf{v}]_{i}}\Phi_{[\mathbf{v}]_{i}\backslash\{t\}}(x_t)\Phi_{[\mathbf{v}]_{\tau i}}(-x_t)\frac{\hbar(x_t+\frac{(n-i)\hbar}{2}+\frac{\hbar}{4})}{u+x_t+\frac{n-i}{2}\hbar+\frac{\hbar}{4}}\\
    &+\sum\limits_{s\in[\mathbf{v}]_{\tau i}}\Phi_{[\mathbf{v}]_{i}}(-x_s)\Phi_{[\mathbf{v}]_{\tau i}\backslash\{s\}}(x_s)\frac{\hbar(x_s-\frac{(n-i)\hbar}{2}-\frac{\hbar}{4})}{u-x_s+\frac{n-i}{2}\hbar+\frac{\hbar}{4}}\bigg)\\
    &+\frac{2\hbar}{u+v}\bigg(\sum\limits_{t\in[\mathbf{v}]_{i}}\Phi_{[\mathbf{v}]_{i}\backslash\{t\}}(x_t)\Phi_{[\mathbf{v}]_{\tau i}}(-x_t)\frac{\hbar(x_t+\frac{(n-i)\hbar}{2}+\frac{\hbar}{4})}{v-x_t-\frac{n-i}{2}\hbar-\frac{\hbar}{4}}\\
    &+\sum\limits_{s\in[\mathbf{v}]_{\tau i}}\Phi_{[\mathbf{v}]_{i}}(-x_s)\Phi_{[\mathbf{v}]_{\tau i}\backslash\{s\}}(x_s)\frac{\hbar(x_s-\frac{(n-i)\hbar}{2}-\frac{\hbar}{4})}{v+x_s-\frac{n-i}{2}\hbar-\frac{\hbar}{4}}\bigg)\\
    =&-\frac{2\hbar}{u+v}\bigg((uH_{i,\mathbf{v}}(u))^\circ-((-v)H_{i,\mathbf{v}}(-v))^\circ\bigg)\\
    =&-\frac{2\hbar}{u+v}(uH^\circ_{i,\mathbf{v}}(u)+vH^\circ_{\tau i,\mathbf{v}}(v)).
\end{align*}
Here the second equality follows from Lemma \ref{lem:Hexp}, while the last one follows from \eqref{equ:trundivbyz}.

Finally, let us check the case $j=\tau i$ and $c_{i,j}\neq 0$. Then $(i,j)=(n,n+1)$ or $(i,j)=(n+1,n)$. By Lemma \ref{lem:inv}, it is enough to check the first case. 

By definition,
\begin{align*}
    &(\mathbf{B}_{n}(u)\mathbf{B}_{n+1}(v)f)(x_{\mathbf{[v]}})\\
    =&-\sum\limits_{t\in[\mathbf{v}]_{n}}\sum\limits_{s\in[\mathbf{v}]_{n+1}\cup\{t,t'\}}\frac{\hbar}{u+x_t+\frac{\hbar}{4}}\frac{\hbar}{v+x_s-\frac{\hbar}{4}}\Phi_{[\mathbf{v}]_{n}\backslash\{t\}}(x_t)\Phi_{[\mathbf{v}]_{n+1}\cup\{t,t'\}\backslash\{s\}}(x_s)\cdot f(x_{\tau_s^+\tau_t^+\mathbf{[v]}}),
\end{align*}
and 
\begin{align*}
    &(\mathbf{B}_{n+1}(v)\mathbf{B}_{n}(u)f)(x_{\mathbf{[v]}})\\
    =&-\sum\limits_{s\in[\mathbf{v}]_{n+1}}\sum\limits_{t\in[\mathbf{v}]_{n}\cup\{s'\}}\frac{\hbar}{u+x_t+\frac{\hbar}{4}}\frac{\hbar}{v+x_s-\frac{\hbar}{4}}\Phi_{[\mathbf{v}]_{n}\cup\{s'\}\backslash\{t\}}(x_t)\Phi_{[\mathbf{v}]_{n+1}\backslash\{s\}}(x_s)\cdot f(x_{\tau_t^+\tau_s^+\mathbf{[v]}}).
\end{align*}
Hence, for $t\in [\mathbf{v}]_{n}$ and $s\in[\mathbf{v}]_{n+1}$, the coefficient of 
\[\frac{\hbar}{u+x_t+\frac{\hbar}{4}}\frac{\hbar}{v+x_s-\frac{\hbar}{4}}\Phi_{[\mathbf{v}]_{n}\cup\{s'\}\backslash\{t\}}(x_t)\Phi_{[\mathbf{v}]_{n+1}\backslash\{s\}}(x_s)\cdot f(x_{\tau_t^+\tau_s^+\mathbf{[v]}})\]
in the left hand side of \eqref{equ:BBij} is
\begin{align*}
    &-(u-v+\frac{\hbar}{2})\varphi(x_t-x_s)+(u-v-\frac{\hbar}{2})+(u+x_t+\frac{\hbar}{4})\varphi(x_t-x_s)\\
    &-(v+x_s-\frac{\hbar}{4})\varphi(x_t-x_s)-(u+x_t+\frac{\hbar}{4})+(v+x_s-\frac{\hbar}{4})=0.
\end{align*}
Similarly, for $t\in [\mathbf{v}]_{n}$, the coefficient of 
\[\frac{\hbar}{u+x_t+\frac{\hbar}{4}}\frac{\hbar}{v+x_t-\frac{\hbar}{4}}\Phi_{[\mathbf{v}]_{n}\backslash\{t\}}(x_t)\Phi_{[\mathbf{v}]_{n+1}\cup\{t'\}}(x_t)\cdot f(x_{\tau_t^+\tau_t^+\mathbf{[v]}})\]
in the left hand side of \eqref{equ:BBij} is
\[-(u-v+\frac{\hbar}{2})+(u+x_t+\frac{\hbar}{4})-(v+x_t-\frac{\hbar}{4})=0.\]
Hence, the coefficient of $f(x_{\mathbf{[v]}})$ in the left hand side of \eqref{equ:BBij} is 
\begin{align*}
    &-(u-v+\frac{\hbar}{2})\sum\limits_{t\in[\mathbf{v}]_{n}}\frac{\hbar}{u+x_t+\frac{\hbar}{4}}\frac{\hbar}{v-x_t-\frac{\hbar}{4}}\Phi_{[\mathbf{v}]_{n}\backslash\{t\}}(x_t)\Phi_{[\mathbf{v}]_{n+1}\cup\{t\}}(-x_t)\\
    &+(u-v-\frac{\hbar}{2})\sum\limits_{s\in[\mathbf{v}]_{n+1}}\frac{\hbar}{u-x_s+\frac{\hbar}{4}}\frac{\hbar}{v+x_s-\frac{\hbar}{4}}\Phi_{[\mathbf{v}]_{n}}(-x_s)\Phi_{[\mathbf{v}]_{n+1}\backslash\{s\}}(x_s)\\
    &+\hbar\sum\limits_{t\in[\mathbf{v}]_{n}}\frac{\hbar}{v-x_t-\frac{\hbar}{4}}\Phi_{[\mathbf{v}]_{n}\backslash\{t\}}(x_t)\Phi_{[\mathbf{v}]_{n+1}\cup\{t\}}(-x_t)\\
    &-\hbar\sum\limits_{t\in[\mathbf{v}]_{n}}\frac{\hbar}{u+x_t+\frac{\hbar}{4}}\Phi_{[\mathbf{v}]_{n}\backslash\{t\}}(x_t)\Phi_{[\mathbf{v}]_{n+1}\cup\{t\}}(-x_t)\\
    &-\hbar\sum\limits_{s\in[\mathbf{v}]_{n+1}}\frac{\hbar}{v+x_s-\frac{\hbar}{4}}\Phi_{[\mathbf{v}]_{n}}(-x_s)\Phi_{[\mathbf{v}]_{n+1}\backslash\{s\}}(x_s)\\
    &+\hbar\sum\limits_{s\in[\mathbf{v}]_{n+1}}\frac{\hbar}{u-x_s+\frac{\hbar}{4}}\Phi_{[\mathbf{v}]_{n}}(-x_s)\Phi_{[\mathbf{v}]_{n+1}\backslash\{s\}}(x_s)\\
    =&\frac{2\hbar}{u+v}\bigg(\sum\limits_{t\in[\mathbf{v}]_{n}}\frac{\hbar(x_t+\frac{\hbar}{2})}{u+x_t+\frac{\hbar}{4}}\Phi_{[\mathbf{v}]_{n}\backslash\{t\}}(x_t)\Phi_{[\mathbf{v}]_{n+1}}(-x_t)\\
    &+\sum\limits_{s\in[\mathbf{v}]_{n+1}}\frac{\hbar(x_s-\frac{\hbar}{2})}{u-x_s+\frac{\hbar}{4}}\Phi_{[\mathbf{v}]_{n}}(-x_s)\Phi_{[\mathbf{v}]_{n+1}\backslash\{s\}}(x_s)\bigg)\\
    &+\frac{2\hbar}{u+v}\bigg(\sum\limits_{t\in[\mathbf{v}]_{n}}\frac{\hbar(x_t+\frac{\hbar}{2})}{v-x_t-\frac{\hbar}{4}}\Phi_{[\mathbf{v}]_{n}\backslash\{t\}}(x_t)\Phi_{[\mathbf{v}]_{n+1}}(-x_t)\\
    &+\sum\limits_{s\in[\mathbf{v}]_{n+1}}\frac{\hbar(x_s-\frac{\hbar}{2})}{v+x_s-\frac{\hbar}{4}}\Phi_{[\mathbf{v}]_{n}}(-x_s)\Phi_{[\mathbf{v}]_{n+1}\backslash\{s\}}(x_s)\bigg)\\
    =&-\frac{2\hbar}{u+v}\bigg((uH_{n,\mathbf{v}}(u))^\circ-((-v)H_{n,\mathbf{v}}(-v))^\circ\bigg)\\
    =&-\frac{2\hbar}{u+v}\bigg(uH^\circ_{n,\mathbf{v}}(u)+vH^\circ_{n,\mathbf{v}}(-v)\bigg).
\end{align*}
Here the second equality follows from Lemma \ref{lem:Hexp}, while the last one follows from \eqref{equ:trundivbyz}.

\subsection{Relation \eqref{equ:hibj}} 
We need to check
\begin{align}\label{equ:hibjche}
&
\left(u^2-v^2-\frac{c_{i j} c_{\tau i, j}}{4} \hbar^2\right)\left[h_i(u), b_j(v)\right]-\bigg(\frac{c_{i j}-c_{\tau i, j}}{2} \hbar u +\frac{c_{i j}+c_{\tau i, j}}{2} \hbar v\bigg)\left\{h_i(u), b_j(v)\right\}\\
&\notag
+\hbar\left[h_i(u), b_{j, 1}\right] 
+\hbar v\left[h_i(u), b_{j, 0}\right]+\frac{c_{i j}+c_{\tau i, j}}{2} \hbar^2\left\{h_i(u), b_{j, 0}\right\}=0.
\end{align}
Since $H_{\tau i}(u)=H_i(-u)$, if \eqref{equ:hibjche} holds for $h_i(u)$, it also holds for $h_{\tau i}(u)$. Hence, we can assume $1\leq i\leq n$. On the other hand, by Lemma \ref{lem:inv}, we can also assume $1\leq j\leq n$. By the definition of the operators, it is enough to study the cases: $j=i, j=i-1$, and $j=i+1$.

\subsubsection{The case $j=i+1\leq n$}

By definition,
\begin{align*}
    &(\mathbf{B}_{j}(v)\mathbf{H}_{j-1}(u)f)(x_{\mathbf{[v]}})\\
    =&\sqrt{-1}\sum\limits_{t\in[\mathbf{v}]_{j}}\frac{\hbar}{v+x_t+\frac{n-j}{2}\hbar+\frac{\hbar}{4}}\Phi_{[\mathbf{v}]_{j}\backslash\{t\}}(x_t)H_{j-1,\mathbf{v}}(u) \cdot f(x_{\tau_t^+\mathbf{[v]}})\\
    &\cdot \varphi(-x_t-u-\frac{n-j}{2}\hbar-\frac{3\hbar}{4})^{-1}.
\end{align*}
Then \eqref{equ:hibjche} holds since
\begin{align*}
    &\bigg((u^2-v^2)-(v+x_t+\frac{n-j}{2}\hbar+\frac{\hbar}{4})(x_t+\frac{n-j}{2}\hbar+\frac{\hbar}{4})+v(v+x_t+\frac{n-j}{2}\hbar+\frac{\hbar}{4})\bigg)\\
    &\cdot (1-\varphi(-x_t-u-\frac{n-j}{2}\hbar-\frac{3\hbar}{4})^{-1})\\
    &+\bigg(\frac{\hbar (u+v)}{2}-\frac{\hbar}{2}(v+x_t+\frac{n-j}{2}\hbar+\frac{\hbar}{4})\bigg)\\
    &\cdot (1+\varphi(-x_t-u-\frac{n-j}{2}\hbar-\frac{3\hbar}{4})^{-1})=0.
\end{align*}

\subsubsection{The case $j=i\leq n$}

Let us first assume $j=i\leq n-1$. Then
\begin{align*}
    &(\mathbf{B}_{j}(v)\mathbf{H}_{j}(u)f)(x_{\mathbf{[v]}})\\
    =&\sqrt{-1}\sum\limits_{t\in[\mathbf{v}]_{j}}\frac{\hbar}{v+x_t+\frac{n-j}{2}\hbar+\frac{\hbar}{4}}\Phi_{[\mathbf{v}]_{j}\backslash\{t\}}(x_t)H_{j,\mathbf{v}}(u)\cdot f(x_{\tau_t^+\mathbf{[v]}})\\
    &\cdot \varphi(-x_t-u-\frac{n-j}{2}\hbar-\frac{\hbar}{4})\varphi(x_t+u+\frac{n-j}{2}\hbar+\frac{\hbar}{4})^{-1},
\end{align*}
and \eqref{equ:hibjche} holds since
\begin{align*}
    &\bigg((u^2-v^2)-(v+x_t+\frac{n-j}{2}\hbar+\frac{\hbar}{4})(x_t+\frac{n-j}{2}\hbar+\frac{\hbar}{4})+v(v+x_t+\frac{n-j}{2}\hbar+\frac{\hbar}{4})\bigg)\\
    &\cdot \bigg(1-\varphi(-x_t-u-\frac{n-j}{2}\hbar-\frac{\hbar}{4})\varphi(x_t+u+\frac{n-j}{2}\hbar+\frac{\hbar}{4})^{-1}\bigg)\\
    &+\bigg(-\hbar (u+v)+\hbar(v+x_t+\frac{n-j}{2}\hbar+\frac{\hbar}{4})\bigg)\\
    &\bigg(1+\varphi(-x_t-u-\frac{n-j}{2}\hbar-\frac{\hbar}{4})\varphi(x_t+u+\frac{n-j}{2}\hbar+\frac{\hbar}{4})^{-1}\bigg)=0.
\end{align*}

Now let us consider the case $j=i=n$.
Then
\begin{align*}
    &(\mathbf{B}_{n}(v)\mathbf{H}_{n}(u)f)(x_{\mathbf{[v]}})\\
    =&\sqrt{-1}\sum\limits_{t\in[\mathbf{v}]_{n}}\frac{\hbar}{v+x_t+\frac{\hbar}{4}}\Phi_{[\mathbf{v}]_{n}\backslash\{t\}}(x_t)H_{n,\mathbf{v}}(u)\cdot f(x_{\tau_t^+\mathbf{[v]}})\\
    &\cdot \varphi(x_t-u-\frac{\hbar}{4})\varphi(-x_t-u-\frac{\hbar}{4})\varphi(x_t+u+\frac{\hbar}{4})^{-1},
\end{align*}
and \eqref{equ:hibjche} holds since
\begin{align*}
    &\bigg((u^2-v^2+\frac{\hbar^2}{2})-(v+x_t+\frac{\hbar}{4})(x_t+\frac{\hbar}{4})+v(v+x_t+\frac{\hbar}{4})\bigg)\\
    &\cdot \bigg(1-\varphi(x_t-u-\frac{\hbar}{4})\varphi(-x_t-u-\frac{\hbar}{4})\varphi(x_t+u+\frac{\hbar}{4})^{-1}\bigg)\\
    &+\bigg(-\hbar (\frac{3u}{2}+\frac{v}{2})+\frac{\hbar}{2}(v+x_t+\frac{\hbar}{4})\bigg)\\
    &\bigg(1+\varphi(x_t-u-\frac{\hbar}{4})\varphi(-x_t-u-\frac{\hbar}{4})\varphi(x_t+u+\frac{\hbar}{4})^{-1}\bigg)=0.
\end{align*}

\subsubsection{The case $j+1=i\leq n$}

As above, 
\begin{align*}
    &(\mathbf{B}_{j}(v)\mathbf{H}_{j+1}(u)f)(x_{\mathbf{[v]}})\\
    =&\sqrt{-1}\sum\limits_{t\in[\mathbf{v}]_{j}}\frac{\hbar}{v+x_t+\frac{n-j}{2}\hbar+\frac{\hbar}{4}}\Phi_{[\mathbf{v}]_{j}\backslash\{t\}}(x_t)H_{j+1,\mathbf{v}}(u)\cdot f(x_{\tau_t^+\mathbf{[v]}})\\
    &\cdot  \varphi(x_t+u+\frac{n-j}{2}\hbar-\frac{\hbar}{4}),
\end{align*}
and \eqref{equ:hibjche} holds since
\begin{align*}
    &\bigg((u^2-v^2)-(v+x_t+\frac{n-j}{2}\hbar+\frac{\hbar}{4})(x_t+\frac{n-j}{2}\hbar+\frac{\hbar}{4})+v(v+x_t+\frac{n-j}{2}\hbar+\frac{\hbar}{4})\bigg)\\
    &\cdot (1-\varphi(x_t+u+\frac{n-j}{2}\hbar-\frac{\hbar}{4}))\\
    &+\bigg(\frac{\hbar (u+v)}{2}-\frac{\hbar}{2}(v+x_t+\frac{n-j}{2}\hbar+\frac{\hbar}{4})\bigg)(1+\varphi(x_t+u+\frac{n-j}{2}\hbar-\frac{\hbar}{4}))=0.
\end{align*}

\subsection{Relation \eqref{equ:bibjcij0}}\label{sec:bibjcij0}

In this section, we check
\begin{align*}
    (u+v)[b_i(u),b_j(v)]=\delta_{\tau i,j}\hbar(h_j(v)-h_i(u)),\qquad c_{ij}=0.
\end{align*}
Since $c_{ij}=0$, it is obvious to see that if $j\neq \tau i$, $\mathbf{B}_i(u)$ and $\mathbf{B}_j(v)$ commute. Therefore, we only need to show the case when $j=\tau i$.
By definition, if $j=\tau i$ and $c_{i,\tau i}=0$,
\begin{align*}
    &(\mathbf{B}_{i}(u)\mathbf{B}_{\tau i}(v)f)(x_{\mathbf{[v]}})\\
    =&-\sum\limits_{t\in[\mathbf{v}]_{i}}\sum\limits_{s\in[\mathbf{v}]_{\tau i}\cup\{t'\}}\frac{\hbar}{u+x_t+\frac{n-i}{2}\hbar+\frac{\hbar}{4}}\frac{\hbar}{v+x_s-\frac{n-i}{2}\hbar-\frac{\hbar}{4}}\Phi_{[\mathbf{v}]_{i}\backslash\{t\}}(x_t)\Phi_{[\mathbf{v}]_{\tau i}\cup\{t'\}\backslash\{s\}}(x_s)\cdot f(x_{\tau_s^+\tau_t^+\mathbf{[v]}}),
\end{align*}
and 
\begin{align*}
    &(\mathbf{B}_{\tau i}(v)\mathbf{B}_{i}(u)f)(x_{\mathbf{[v]}})\\
    =&-\sum\limits_{s\in[\mathbf{v}]_{\tau i}}\sum\limits_{t\in[\mathbf{v}]_{ i}\cup\{s'\}}\frac{\hbar}{u+x_t+\frac{n-i}{2}\hbar+\frac{\hbar}{4}}\frac{\hbar}{v+x_s-\frac{n-i}{2}\hbar-\frac{\hbar}{4}}\Phi_{[\mathbf{v}]_{i}\cup\{s'\}\backslash\{t\}}(x_t)\Phi_{[\mathbf{v}]_{\tau i}\backslash\{s\}}(x_s)\cdot f(x_{\tau_s^+\tau_t^+\mathbf{[v]}}).
\end{align*}
Therefore, 
\begin{align*}
    &(u+v)(\mathbf{B}_{i}(u)\mathbf{B}_{\tau i}(v)f-\mathbf{B}_{\tau i}(v)\mathbf{B}_{i}(u)f)(x_{\mathbf{[v]}})\\
    =&-(u+v)\sum\limits_{t\in[\mathbf{v}]_{i}}\frac{\hbar}{u+x_t+\frac{n-i}{2}\hbar+\frac{\hbar}{4}}\frac{\hbar}{v-x_t-\frac{n-i}{2}\hbar-\frac{\hbar}{4}}\Phi_{[\mathbf{v}]_{i}\backslash\{t\}}(x_t)\Phi_{[\mathbf{v}]_{\tau i}}(-x_t)\cdot f(x_{\mathbf{[v]}})\\
    &+(u+v)\sum\limits_{s\in[\mathbf{v}]_{\tau i}}\frac{\hbar}{u-x_s+\frac{n-i}{2}\hbar+\frac{\hbar}{4}}\frac{\hbar}{v+x_s-\frac{n-i}{2}\hbar-\frac{\hbar}{4}}\Phi_{[\mathbf{v}]_{i}}(-x_s)\Phi_{[\mathbf{v}]_{\tau i}\backslash\{s\}}(x_s)\cdot f(x_{\mathbf{[v]}})\\
    =&\hbar f(x_{\mathbf{[v]}})\bigg(\sum\limits_{s\in[\mathbf{v}]_{\tau i}}\frac{\hbar}{v+x_s-\frac{n-i}{2}\hbar-\frac{\hbar}{4}}\Phi_{[\mathbf{v}]_{i}}(-x_s)\Phi_{[\mathbf{v}]_{\tau i}\backslash\{s\}}(x_s)\\
    &-\sum\limits_{t\in[\mathbf{v}]_{i}}\frac{\hbar}{v-x_t-\frac{n-i}{2}\hbar-\frac{\hbar}{4}}\Phi_{[\mathbf{v}]_{i}\backslash\{t\}}(x_t)\Phi_{[\mathbf{v}]_{\tau i}}(-x_t)\bigg)\\
    &-\hbar f(x_{\mathbf{[v]}})\bigg(\sum\limits_{t\in[\mathbf{v}]_{i}}\frac{\hbar}{u+x_t+\frac{n-i}{2}\hbar+\frac{\hbar}{4}}\Phi_{[\mathbf{v}]_{i}\backslash\{t\}}(x_t)\Phi_{[\mathbf{v}]_{\tau i}}(-x_t)\\
    &-\sum\limits_{s\in[\mathbf{v}]_{\tau i}}\frac{\hbar}{u-x_s+\frac{n-i}{2}\hbar+\frac{\hbar}{4}}\Phi_{[\mathbf{v}]_{i}}(-x_s)\Phi_{[\mathbf{v}]_{\tau i}\backslash\{s\}}(x_s)\bigg)\\
    =&\hbar\bigg(\mathbf{H}_{\tau i}(v)f-\mathbf{H}_i(u)f\bigg)(x_{\mathbf{[v]}}).
\end{align*}
Here the last equality follows from the first equation in Lemma \ref{lem:Hexp}.

\subsection{Serre relation}
In this section, we check \eqref{equ:fserre}. The other Serre relation \eqref{equ:fserreij} can be checked in an easier way, so we omit it.
We need to show 
\[[b_{n,0},[b_{n,0},b_{n+1,0}]]=4b_{n,0},\textit{ and }
        [b_{n+1,0},[b_{n+1,0},b_{n,0}]]=4b_{n+1,0}.\]
By Lemma \ref{lem:inv}, we only need to check the first one. Recall
\begin{equation*}
    (B_{n,0}f)(x_{[\mathbf{v}]}):=\sqrt{-1}\sum_{r\in [\mathbf{v}]_{n}}\Phi_{[\mathbf{v}]_{n}\setminus\{r\}}(x_r)\cdot  f(x_{\tau^+_r[\mathbf{v}]}),
\end{equation*}
and 
\begin{equation*}
    (B_{n+1,0}f)(x_{[\mathbf{v}]}):=\sqrt{-1}\sum_{s\in [\mathbf{v}]_{n+1}}\Phi_{[\mathbf{v}]_{n+1}\setminus\{s\}}(x_s)\cdot  f(x_{\tau^+_s[\mathbf{v}]}).
\end{equation*}
We need to show
\begin{align}\label{equ:serreB}
    &\sqrt{-1}(B_{n,0}B_{n,0}B_{n+1,0}f-2B_{n,0}B_{n+1,0}B_{n,0}f+B_{n+1,0}B_{n,0}B_{n,0}f)(x_{[\mathbf{v}]})\\
    =&\sqrt{-1}(4B_{n,0}f)(x_{[\mathbf{v}]}).\notag
\end{align}
By definition,
\begin{align*}
    &\sqrt{-1}(B_{n,0}B_{n,0}B_{n+1,0}f)(x_{[\mathbf{v}]})\\
    =&\sum_{s\in [\mathbf{v}]_{n}}\sum_{t\in [\mathbf{v}]_{n}\setminus\{s\}}\sum_{r\in [\mathbf{v}]_{n+1}\cup\{s,t,s',t'\}}\Phi_{[\mathbf{v}]_{n}\setminus\{s\}}(x_s)\Phi_{[\mathbf{v}]_{n}\setminus\{s,t\}}(x_t)\\
    &\Phi_{[\mathbf{v}]_{n+1}\cup\{s,t,s',t'\}\setminus\{r\}}(x_r)\cdot  f(x_{\tau_r^+\tau_t^+\tau^+_s[\mathbf{v}]}),
\end{align*}
\begin{align*}
    &\sqrt{-1}(B_{n,0}B_{n+1,0}B_{n,0}f)(x_{[\mathbf{v}]})\\
    =&\sum_{s\in [\mathbf{v}]_{n}}\sum_{r\in [\mathbf{v}]_{n+1}\cup\{s,s'\}}\sum_{t\in [\mathbf{v}]_{n}\cup\{r'\}\setminus\{s\}}\Phi_{[\mathbf{v}]_{n}\setminus\{s\}}(x_s)\Phi_{[\mathbf{v}]_{n}\cup\{r'\}\setminus\{s,t\}}(x_t)\\
    &\Phi_{[\mathbf{v}]_{n+1}\cup\{s,s'\}\setminus\{r\}}(x_r)\cdot  f(x_{\tau_t^+\tau_r^+\tau^+_s[\mathbf{v}]}),
\end{align*}
and 
\begin{align*}
    &\sqrt{-1}(B_{n+1,0}B_{n,0}B_{n,0}f)(x_{[\mathbf{v}]})\\
    =&\sum_{r\in [\mathbf{v}]_{n+1}}\sum_{s\in [\mathbf{v}]_{n}\cup\{r'\}}\sum_{t\in [\mathbf{v}]_{n}\cup\{r'\}\setminus\{s\}}\Phi_{[\mathbf{v}]_{n}\cup\{r'\}\setminus\{s\}}(x_s)\Phi_{[\mathbf{v}]_{n}\cup\{r'\}\setminus\{s,t\}}(x_t)\\
    &\Phi_{[\mathbf{v}]_{n+1}\setminus\{r\}}(x_r)\cdot  f(x_{\tau_t^+\tau_s^+\tau_r^+[\mathbf{v}]}).
\end{align*}

Therefore, for $s\neq t\in [\mathbf{v}]_n$ and $r\in [\mathbf{v}]_{n+1}$, the coefficient of $f(x_{\tau_t^+\tau_s^+\tau_r^+[\mathbf{v}]})$ on the left hand side of \eqref{equ:serreB} is
\begin{align*}
    &\Phi_{[\mathbf{v}]_{n}\setminus\{s,t\}}(x_s)\Phi_{[\mathbf{v}]_{n}\setminus\{s,t\}}(x_t)\Phi_{[\mathbf{v}]_{n+1}\setminus\{r\}}(x_r)\varphi(-x_r-x_s)\varphi(-x_r-x_t)\\
    &\cdot \bigg(\varphi(x_s-x_r)\varphi(x_t-x_r)\varphi(x_t-x_s)+\varphi(x_s-x_r)\varphi(x_t-x_r)\varphi(x_s-x_t)\\
    &-2\varphi(x_t-x_s)\varphi(x_s-x_r)-2\varphi(x_s-x_t)\varphi(x_t-x_r)+\varphi(x_t-x_s)+\varphi(x_s-x_t)\bigg)\\
    =&0.
\end{align*}
Morevoer, for $s\neq t\in [\mathbf{v}]_n$, the coefficient of \[\Phi_{[\mathbf{v}]_{n}\setminus\{s\}}(x_s)\Phi_{[\mathbf{v}]_{n}\cup\{s'\}\setminus\{s,t\}}(x_t)\Phi_{[\mathbf{v}]_{n+1}\cup\{s'\}}(x_s)f(x_{\tau_{s}^+\tau_t^+\tau^+_s[\mathbf{v}]})\] in the left hand side of \eqref{equ:serreB} is
\begin{align*}
    &\varphi(x_t-x_s)+\varphi(x_s-x_t)-2=0.
\end{align*}

Therefore, the terms on the left hand side of \eqref{equ:serreB} are all of the form $\Phi_{[\mathbf{v}]_{n}\setminus\{s\}}(x_s)\cdot f(x_{\tau^+_s[\mathbf{v}]})$ for some $s \in [\mathbf{v}]_{n}$, and its coefficient is
\begin{align*}
    &\sum_{t\in [\mathbf{v}]_{n}\setminus\{s\}}\Phi_{[\mathbf{v}]_{n}\setminus\{t\}}(x_t)\Phi_{[\mathbf{v}]_{n+1}\cup\{t,s\}}(-x_t)+\sum_{t\in [\mathbf{v}]_{n}\setminus\{s\}}\Phi_{[\mathbf{v}]_{n}\setminus\{s,t\}}(x_t)\Phi_{[\mathbf{v}]_{n+1}\cup\{s,t,s'\}}(-x_t)\\
    &-2\Phi_{[\mathbf{v}]_{n}\setminus\{s\}}(-x_s)\Phi_{[\mathbf{v}]_{n+1}\cup\{s'\}}(x_s)\\
    &-2\sum_{t\in [\mathbf{v}]_{n}}\Phi_{[\mathbf{v}]_{n}\setminus\{t\}}(x_t)\Phi_{[\mathbf{v}]_{n+1}\cup\{t\}}(-x_t)-2\sum_{r\in [\mathbf{v}]_{n+1}}\Phi_{[\mathbf{v}]_{n}\setminus\{s\}}(-x_r)\Phi_{[\mathbf{v}]_{n+1}\cup\{s,s'\}\setminus\{r\}}(x_r)\\
    &+\sum_{r\in [\mathbf{v}]_{n+1}}\Phi_{[\mathbf{v}]_{n}}(-x_r)\Phi_{[\mathbf{v}]_{n+1}\setminus\{r\}}(x_r)+\sum_{r\in [\mathbf{v}]_{n+1}}\varphi(-x_r-x_s)\Phi_{[\mathbf{v}]_{n}\setminus\{s\}}(-x_r)\Phi_{[\mathbf{v}]_{n+1}\setminus\{r\}}(x_r)\\
    =&2\sum_{t\in [\mathbf{v}]_{n}\setminus\{s\}}\Phi_{[\mathbf{v}]_{n}\setminus\{s,t\}}(x_t)\Phi_{[\mathbf{v}]_{n+1}}(-x_t)\frac{\hbar(\hbar+2x_t)}{x_t^2-x_s^2}\\
    &-2\sum_{r\in [\mathbf{v}]_{n+1}}\Phi_{[\mathbf{v}]_{n}\setminus\{s\}}(-x_r)\Phi_{[\mathbf{v}]_{n+1}\setminus\{r\}}(x_r)\frac{\hbar(\hbar-2x_r)}{x_r^2-x_s^2}\\
    &-2\Phi_{[\mathbf{v}]_{n}\setminus\{s\}}(-x_s)\Phi_{[\mathbf{v}]_{n+1}\cup\{s'\}}(x_s)-2\Phi_{[\mathbf{v}]_{n}\setminus\{s\}}(x_s)\Phi_{[\mathbf{v}]_{n+1}\cup\{s\}}(-x_s)\\
    =&2\Res(f(x),\infty)\\
    =&-4,
\end{align*}
where the second equation follows from the residue theorem for the function
\begin{align*}
    f(x):=&\Phi_{[\mathbf{v}]_{n}\setminus\{s\}}(x)\Phi_{[\mathbf{v}]_{n+1}}(-x)\frac{\hbar+2x}{(x+x_s)(x-x_s)}\\
    =&\prod_{t\in [\mathbf{v}]_{n}\setminus\{s\}}(1+\frac{\hbar}{x_t-x})\prod_{r\in [\mathbf{v}]_{n+1}}(1+\frac{\hbar}{x+x_r})\cdot\frac{\hbar+2x}{(x+x_s)(x-x_s)}.
\end{align*}
Hence, \eqref{equ:serreB} holds, and this concludes the proof of Theorem \ref{thm:polyrep}.

\bibliographystyle{alpha}
\bibliography{Yangian}

\end{document}